\documentclass[a4paper,12pt]{amsart}
\usepackage{mathdots, tikz}

\newtheorem{thm}{Theorem}[section]
\newtheorem{lem}[thm]{Lemma}
\newtheorem{prp}[thm]{Proposition}
\newtheorem{cor}[thm]{Corollary}
\theoremstyle{definition}
\newtheorem{dfn}[thm]{Definition}
\newtheorem{ntn}[thm]{Notation}
\theoremstyle{remark}

\numberwithin{equation}{section}

\newcommand{\Ad}{\operatorname{Ad}}
\newcommand{\aua}{\mathbin{\sim_{\operatorname{au}}}}
\newcommand{\clsp}{\overline{\lsp}}
\newcommand{\coker}{\operatorname{coker}}
\newcommand{\FK}{\operatorname{FK}}
\newcommand{\id}{\operatorname{id}}
\newcommand{\image}{\operatorname{Im}}
\newcommand{\lsp}{\operatorname{span}}
\newcommand{\Prim}{\operatorname{Prim}}
\newcommand{\nucdim}{\operatorname{dim_{\operatorname{nuc}}}}

\newcommand{\CC}{\mathbb{C}}
\newcommand{\NN}{\mathbb{N}}
\newcommand{\OO}{\mathbb{O}}
\newcommand{\TT}{\mathbb{T}}
\newcommand{\ZZ}{\mathbb{Z}}

\newcommand{\Bb}{\mathcal{B}}
\newcommand{\Kk}{\mathcal{K}}
\newcommand{\Mm}{\mathcal{M}}
\newcommand{\Oo}{\mathcal{O}}
\newcommand{\Tt}{\mathcal{T}}
\newcommand{\Zz}{\mathcal{Z}}

	\title{The nuclear dimension of graph $C^{*}$-algebras}
	
	\author{Efren Ruiz}
        \address{Department of Mathematics\\University of Hawaii,
Hilo\\200 W. Kawili St.\\
Hilo, Hawaii\\
96720-4091 USA}
        \email{ruize@hawaii.edu}

        \author{Aidan Sims}
       \address{School of Mathematics and Applied Statistics \\
Faculty of Engineering and Information Sciences \\
Room 152, Building 39C \\
University of Wollongong NSW 2522}
	\email{asims@uow.edu.au}

        \author{Mark Tomforde}
        \address{Department of Mathematics\\University of Houston\\
Houston, Texas\\
77204- 3008, USA}
        \email{tomforde@math.uh.edu}
        \date{\today}
	
\keywords{Graph algebras, nuclear dimension}
\subjclass[2010]{Primary: 46L35}
\thanks{This research was supported by the Australian Research Council and by the Simons
Foundation (Collaboration Grant \#279369 to Efren Ruiz and Collaboration Grant \#210035 to Mark Tomforde).
Part of the project was carried out when the authors met at the Banff International
Research Station for the workshop \emph{Graph algebras: Bridges between graph C*-algebras and Leavitt path algebras} (13w5049).}

\begin{document}

\begin{abstract}
Consider a graph $C^*$-algebra $C^*(E)$ with a purely infinite ideal $I$ (possibly all of
$C^*(E)$) such that $I$ has only finitely many ideals and $C^*(E)/I$ is approximately
finite dimensional. We prove that the nuclear dimension of $C^*(E)$ is~1. If $I$ has
infinitely many ideals, then the nuclear dimension of $C^*(E)$ is either 1~or~2.
\end{abstract}

\maketitle

\section{Introduction}

The point of view that regards $C^*$-algebras as noncommutative topological spaces has
led to a number of notions of topological dimension for $C^*$-algebras (see, for example,
\cite{BrownPedersen:JFA91, KirchbergWinter, Rieffel:PLMS83, Winter:JFA07}). Each of these
captures important $C^*$-algebraic properties, and many of them have played an important
role in the program of classification of $C^*$-algebras by $K$-theoretic data pioneered
by Elliott (see, for example, \cite{Elliott:JA76, Elliott:JRAM93,
ElliottGongEtAl:DMJ96}). In 2010, Winter and Zacharias introduced the nuclear dimension
of a $C^*$-algebra as a noncommutative analogue of topological covering dimension
\cite{WinterZacharias:AM10}. Finite nuclear dimension is closely related to
$\Zz$-stability \cite{Winter:IM12} where $\Zz$ is the Jiang-Su algebra (Toms and Winter
have conjectured that the two are equivalent for the class of simple, separable,
infinite-dimensional, nuclear $C^{*}$-algebras), and so has important implications for
classification theory.

Roughly speaking, the nuclear dimension of a $C^*$-algebra $A$ is the minimum number $d$
for which the identity map on $A$ can be approximately factored, on any finite set of
elements, through a finite-dimensional $C^*$-algebra by a composition $\id \sim
\varphi\circ \psi$ where $\psi$ is a completely positive contraction and $\varphi$
decomposes as a direct sum of $d+1$ completely positive orthogonality-preserving
contractions. As explained in \cite[Section~1]{KirchbergWinter} (see also
\cite{WinterZacharias:AM10}), this relates to covering dimension as follows\footnote{The
second author is indebted for this explanation to a talk by Aaron Tikuisis at the
workshop \emph{Classifying Structures for Operator Algebras and Dynamical Systems} held
at the University of Aberystwyth in September 2013}. Consider a compact Hausdorff space
$X$ and the commutative $C^*$-algebra $C(X)$. Given positive $f_1, \dots, f_n \in C(X)$,
cover the union of their supports with open sets $U_1, \dots, U_h$ so that each $f_i$ is
approximately constant on each $U_j$. Fix a point $x_i$ in each $U_i$ and define $\psi :
C(X) \to \CC^h$ by $\psi(f)_i = f(x_i)$ for all $1 \leq i \leq h$. Define $\varphi :
\CC^h \to C(X)$ by $\varphi(a_1, \dots a_h) = \sum_{i=1}^h a_i u_i$ for some partition of
unity $\{ u_i \}_{i=1}^h$ subordinate to $\mathcal{U} = \{U_1, \dots, U_h\}$. So each
$\varphi \circ \psi(f_i)$ is close to $f_i$ by construction. Partitioning $\mathcal{U}$
into subcollections $\mathcal{U} = \mathcal{U}^0 \sqcup \ldots \sqcup \mathcal{U}^d$ so
that distinct elements of any given $\mathcal{U}^i$ are disjoint gives a decomposition of
$\CC^h$ into $d+1$ direct summands on which $\varphi$ preserves orthogonality. The
smallest $d$ for which we can always do this (to a refinement of $\mathcal{U}$) is the
covering dimension of $X$.

In \cite{WinterZacharias:AM10}, Winter and Zacharias established a number of fundamental
properties of nuclear dimension. They showed that nuclear dimension behaves well with
respect to stabilization, direct sums, tensor products, hereditary subalgebras, direct
limits, and extensions. They also showed that for $2 \leq n < \infty$ the Cuntz algebras
$\Oo_n$ have nuclear dimension 1. To do this, they employed elements of a construction
used in \cite{Kribs} to realise $C^*$-algebras generated by weighted shift operators as
direct limits of Toeplitz-Cuntz algebras. They constucted pairs $\varphi_m$, $\psi_m$
such that each $\psi_m : \Oo_n \to F$ is a completely positive contraction onto a
finite-dimensional $C^*$-algebra, each $\varphi_m : F \to \Oo_n \otimes M_{d_m}$ is a
direct sum of two orthogonality-preserving completely positive contractions, and
$\varphi_m \circ \psi_m \to \kappa_m$ pointwise as $m \to \infty$, where $\kappa_m :
\Oo_n \to \Oo_n \otimes M_{d_m}$ is a homomorphism that induces multiplication by $m$ in
$K$-theory. Since $K_*(\Oo_n) = (\ZZ_{n-1}, 0)$, choosing the $m$ appropriately ensures
that the $\kappa_m$ induce the identity in $K$-theory, and then Kirchberg-Phillips'
classification results show that the $\varphi_m \circ \psi_m$ are asymptotically
approximately conjugate to the identity map. Winter and Zacharias' results about
extensions then show that the Toeplitz-Cuntz algebras $\Tt \Oo_n$ have nuclear dimension
at most~2, and hence $\Oo_\infty$, being a direct limit of the $\Tt \Oo_n$, also has
nuclear dimension at most 2. Finally, using direct-limit decompositions of Kirchberg
algebras and the behavior of nuclear dimension with respect to tensor products, Winter
and Zacharias deduce that every Kirchberg algebra has nuclear dimension at most~5.

More recently, Enders \cite{Enders:xx} developed a technique for showing that
$\Oo_\infty$ in fact has nuclear dimension 1. The rough idea of his argument is to
proceed as in Winter and Zacharias' proof for $\Oo_n$ up to the construction of the
$\varphi_m \circ \psi_m$ that approximate homomorphisms $\kappa_m$ inducing
multiplication by $m$ in $K$-theory. At this point, Enders uses that multiplication by
$-1$ in $K_*(\Oo_\infty)$ is an isomorphism, and so, by Kirchberg-Phillips' results, is
induced by an automorphism $\lambda$ of $\Oo_\infty$. He then constructs approximating
pairs $\Phi_m, \Psi_m$ from the $\varphi_{m+1} \circ \psi_{m+1}$ and $\lambda \circ
\varphi_m \circ \psi_m$ such that the $\Phi_m \circ \Psi_m$ approximate homomorphisms
that induce the identity map in $K$-theory. He then argues using Kirchberg-Phillips'
results again that the $\Phi_m \circ \Psi_m$ are approximately conjugate to the identity.
Enders actually shows that all Kirchberg algebras in the Rosenberg-Schochet bootstrap
category $\mathcal{N}$ with torsion-free $K_1$ have nuclear dimension~1.  Building upon
these results, and the techniques developed in this paper, the first two authors and
S{\o}rensen \cite{RSS} have showed that all UCT Kirchberg algebras have nuclear
dimension~1. Using more direct techniques, Matui and Sato \cite{MatuiSato} have shown
that all Kirchberg algebras (regardless of UCT) have nuclear dimension at most~3; and
Bosa, Brown, Sato, Tikuisis, White and Winter have recently announced that in fact the
exact value is always~1.

Here we address the nuclear dimension of nonsimple graph $C^*$-algebras; our cleanest
results are for the purely infinite situation. Winter and Zacharias' results imply that
if $C^*(E)$ has at most $m$ primitive ideals, then the nuclear dimension of $C^*(E)$ is at most
$6m-1$, and Enders' result improves this bound to $3m-1$. We prove that in fact if $E$ is
a directed graph and $C^*(E)$ is purely infinite and has finitely many ideals, then
$\nucdim(C^*(E)) = 1$. In particular, every Cuntz-Krieger algebra with finitely many
ideals has nuclear dimension~1. A key tool for us is a construction due to Kribs and
Solel \cite{KribsSolel:JAMS07} that generalises Kribs' construction of direct limits of
Toeplitz Cuntz algebras \cite{Kribs} to a construction of direct limits of
Toeplitz-Cuntz-Krieger algebras associated to directed graphs. Using an adaptation of a
direct-limit decomposition of graph $C^*$-algebras due to Jeong and Park
\cite{JeongPark:JFA02}, we deduce that every purely infinite graph $C^*$-algebra has
nuclear dimension at most 2. We then consider graph
$C^*$-algebras of ``mixed" type. We show that if $I$ is an ideal of $C^*(E)$ for which
the quotient is AF, then the extension is quasidiagonal in the sense that $I$ contains an
approximate identity of projections that is asymptotically central in $C^*(E)$. Using
this we deduce that the nuclear dimension of $C^*(E)$ is at most that of $I$. So if $I$
is purely infinite, then the nuclear dimension of $C^*(E)$ is at most~2, and if $I$ is
purely infinite and has finitely many ideals, then the nuclear dimension of $C^*(E)$ is~1
(see \cite{BEMSW} for an alternative, and more direct, approach to finding upper bounds
for the nuclear dimension of a non-simple $\Oo_\infty$-absorbing $C^*$-algebras).

Let $A$ be a separable, nuclear, purely infinite, tight $C^{*}$-algebra over an accordion
space $X$ (meaning that there is a homeomorphism $\psi : \Prim(A) \to X$). Suppose that
$K_{1}(A(x))$ is free and that $A(x)$ belongs to the Rosenberg-Schochet bootstrap
category $\mathcal{N}$ for each $x \in X$. Recent results of Arklint, Bentmann, and
Katsura (\cite{ArklintBentmannKatsura:xx13inv} and
\cite{ArklintBentmannKatsura:xx13range}), show that $A$ is stably isomorphic to $C^*(E)$
for some row-finite graph $E$.  Since nuclear dimension is preserved by stabilization,
these results imply that $A$ has nuclear dimension~1.  Evidence suggests that, more
generally, if $A$ is a separable, nuclear, purely infinite, tight $C^{*}$-algebra over
any finite topological space $X$ and if $K_{1}(A(x))$ is free and $A(x)$ is in
$\mathcal{N}$ for all $x \in X$, then $A$ is a purely infinite graph $C^{*}$-algebra.  If
so, then our results would imply that if $A$ is a separable, nuclear, purely infinite
$C^{*}$-algebra with finitely many ideals such that every simple subquotient of $A$
belongs to $\mathcal{N}$ and has a free abelian $K_1$-group, then $A$ has nuclear
dimension~1.

\medskip

The paper is structured as follows. In Section~\ref{sec:approximation}, given a
row-finite directed graph $E$ with no sinks, and the corresponding sequence of graphs
$E(m)$ of \cite{KribsSolel:JAMS07}, we describe homomorphisms $\tilde{\iota}_m : C^*(E)
\to C^*(E(m))$ constructed by Rout, and adapt the approach of
\cite[Section~7]{WinterZacharias:AM10} to show that the homomorphisms $\tilde{\iota}_m :
C^*(E) \to C^*(E(m))$ approximately factor through direct sums of two order-zero maps
through finite-dimensional $C^*$-algebras. In Section~\ref{sec:reinclusion}, again
following \cite[Section~7]{WinterZacharias:AM10}, we construct homomorphisms $j_m :
C^*(E(m)) \to C^*(E) \otimes \Kk$ and prove that if $C^*(E)$ has finitely many ideals,
then the homomorphisms $j_m \circ \tilde{\iota}_m : C^*(E) \to C^*(E) \otimes \Kk$ induce
multiplication by $m$ in the $K$-groups of every ideal and quotient of $C^*(E)$. In
Section~\ref{sec:enders}, we combine this with heavy machinery of \cite{Kirchberg,
MeyerNest:MJM09, MeyerNest:CJM12} and a technique developed by Enders \cite{Enders:xx} to
prove that $C^*(E)$ has nuclear dimension~1 when it is purely infinite and has finitely
many ideals; we deduce that strongly purely infinite nuclear UCT $C^*$-algebras whose
primitive ideal spaces are finite accordion spaces have nuclear dimension~1 using the
classification results of \cite{ArklintBentmannKatsura:xx13range,
ArklintBentmannKatsura:xx13inv}. In a short section~\ref{sec:consequences}, we use our
main result and a technique of Jeong-Park \cite{JeongPark:JFA02} to see that purely
infinite graph $C^*$-algebras with infinitely many ideals have nuclear dimension at
most~2. In Section~\ref{sec:qdext}, we show that the nuclear dimension of a quasidiagonal
extension $0 \to I \to A \to A/I \to 0$ (see Definition~\ref{dfn:qde}) is the maximum of
those of $I$ and $A/I$, and use this and a result of Gabe to investigate the nuclear
dimension of $C^*(E)$ when it admits a purely infinite ideal $I$ such that $C^*(E)/I$ is
AF.

\section{Approximation by order-zero maps}\label{sec:approximation}

In this section we construct homomorphisms from a graph algebra $C^*(E)$ into related
graph algebras $C^*(E(m))$ that can be approximately factored through finite-dimensional
$C^*$-algebras as sums of two order-zero maps. Our approach closely follows the technique
developed by Winter and Zacharias in \cite[Section~7]{WinterZacharias:AM10} to compute
the nuclear dimension of Cuntz algebras.

For a directed graph $E$, let $E^{*}$ denote the set of all finite paths in $E$, let
$E^{n}$ denote the set of all paths in $E$ of length $n$, and let $E^{<n}$ denote the set
of all paths in $E$ of length strictly less than $n$.  We regard vertices as paths of
length zero.  For $\mu \in E^*$, we define
\begin{align*}
E^{n} \mu := \{ \alpha \mu : \alpha \in E^{n} \text{ and }  r( \alpha ) = s ( \mu ) \} \\
\mu E^{n} := \{ \mu \alpha : \alpha \in E^{n} \text{ and } s( \alpha ) = r ( \mu ) \},
\end{align*}
and we define $\mu E^{*}$, $E^{*}\mu$, $\mu E^{<n}$, and $E^{<n} \mu$ similarly.

Recall that if $E$ is a row-finite directed graph, then its Toeplitz algebra $\Tt C^*(E)$
is the universal $C^*$-algebra generated by mutually orthogonal projections $\{q_v : v
\in E^0\}$ and elements $\{t_e : e \in E^1\}$ such that
\begin{itemize}
\item[(TCK1)] $t^*_e t_e = q_{r(e)}$ for all $e \in E^1$, and
\item[(TCK2)] $q_v \ge \sum_{e \in vE^1} t_e t^*_e$ for each $v \in E^0$.
\end{itemize}
There is a Toeplitz-Cuntz-Krieger $E$-family in $\ell^2(E^*)$ given by $T_e \xi_\mu =
\delta_{r(e), s(\mu)} \xi_{e\mu}$ and $Q_v \xi_\mu = \delta_{v, s(\mu)} \xi_\mu$. Since
$\big(Q_v - \sum_{e \in vE^1} T_e T^*_e\big) \xi_v = \xi_v$, we see that the inequality
in (TCK2) is strict in $\Tt C^*(E)$.

We recall some background from \cite{WinterZacharias:AM10}. A completely positive map
$\varphi : A \to B$ between $C^*$-algebras has \emph{order zero} if, whenever $a,b$ are
positive elements of $A$ with $ab = ba = 0$, we also have $\varphi(a)\varphi(b) =
\varphi(b)\varphi(a) = 0$. A $C^*$-algebra $A$ has nuclear dimension at most $n \ge 0$ if
there is a net $(F_\lambda, \psi_\lambda, \varphi_\lambda)$ such that the $F_\lambda$ are
finite-dimensional $C^*$-algebras, and $\psi_\lambda : A \to F_\lambda$ and
$\varphi_\lambda : F_\lambda \to A$ are completely positive maps satisfying
\begin{enumerate}
\item\label{it:nucdim1} $\varphi_\lambda \circ \psi_\lambda(a) \to a$ for each $a \in A$;
\item\label{it:nucdim2} $\|\psi_\lambda\| \le 1$ for all $\lambda$; and
\item\label{it:nucdim3} each $F_\lambda$ decomposes as $F_\lambda = \bigoplus^n_{i=0}
    F_\lambda^{(i)}$ with each $\varphi_\lambda|_{F_\lambda^{(i)}}$ a completely
    positive contraction of order zero.
\end{enumerate}

Given a countable set $S$, the compact operators on $\ell^{2} ( S )$ will be denoted by
$\Kk_{S}$. Equivalently, $\Kk_S$ is the unique nonzero $C^*$-algebra generated by matrix
units indexed by $S$; that is, elements $\{\theta_{s,t} : s,t \in S\}$ such that
$\theta_{s,t}^* = \theta_{t,s}$ and $\theta_{s,t} \theta_{u,v} = \delta_{t,u}
\theta_{s,v}$.

\begin{lem}\label{lem:compacts}
Let $E$ be a row-finite directed graph. For $\mu \in E^*$, let $\Delta_\mu := t_\mu
t^*_\mu - \sum_{\mu e \in \mu E^1} t_{\mu e} t^*_{\mu e} \in \Tt C^*(E)$. Then
$\Delta_\mu = t_\mu \Delta_{r(\mu)} t^*_\mu$ for each $\mu \in E^*$, and there is an
isomorphism
 from $\clsp\{t_\mu \Delta_v t^*_\nu : v\in E^0\text{ and }\mu,\nu \in E^* v\}$ onto
$\bigoplus_{v \in E^0} \Kk_{E^* v}$ that carries each $t_\mu \Delta_v t^*_\nu$ to
$\theta_{\mu,\nu}$. For each $m \in \NN$, the series $\sum_{\mu \in E^{<m}} \Delta_\mu$
converges strictly to a projection $\Phi_m \in \Mm(\Tt C^*(E) )$. Moreover,
\[
\Phi_m t_\alpha t^*_\beta \Phi_m
	= \sum_{\substack{\tau \in r(\alpha)E^*,\\ |\alpha\tau|,|\beta\tau| < m}}
		t_{\alpha\tau}\Delta_{r(\tau)}t^*_{\beta\tau},
\]
and
\[
\Phi_m \Tt C^*(E) \Phi_m = \clsp\{t_\mu \Delta_{r(\mu)} t^*_\nu : |\mu|, |\nu| < m\}.
\]
\end{lem}
\begin{proof}
We have $t_\mu \Delta_{r(\mu)} t^*_\mu = t_\mu q_{r(\mu)} t_\mu^{*} - t_\mu \sum_{e \in
r(\mu)E^1} t_e t^*_e t_\mu^{*} = \Delta_\mu$, proving the first assertion. We also have
\[
\|t_\mu \Delta_v t_\nu^{*} \|
	\ge \|t_\mu^* t_\mu \Delta_v t_\nu^* t_\nu\| = \|\Delta_v\| = 1
\]
as discussed immediately after the definition of a Toeplitz-Cuntz-Krieger family above.
Hence to establish the desired isomorphism from $\clsp\{t_\mu \Delta_v t^*_\nu : v\in
E^0\text{ and }\mu,\nu \in E^* v\}$ onto $\bigoplus_{v \in E^0} \Kk_{E^* v}$,  it
suffices to show that the $t_\mu \Delta_v t^*_\nu$ are matrix units. Certainly $(t_\mu
\Delta_v t^*_\nu)^* = t_\nu \Delta_v t^*_\mu$ because $\Delta_v$ is a projection.  If
$\mu,\nu \in E^*v$ and $\alpha, \beta \in E^*w$, then
\[
t_\mu \Delta_v t^*_\nu t_\alpha \Delta_w t^*_\beta
	= \begin{cases}
		t_\mu \Delta_v t_{\alpha'} \Delta_w t^*_\beta &\text{ if $\alpha = \nu\alpha'$}\\
		t_\mu \Delta_v t^*_{\nu'} \Delta_w t^*_\beta &\text{ if $\nu = \alpha\nu'$}\\
		0 &\text{ otherwise.}		
	\end{cases}
\]
If $\alpha = \nu\alpha'$ and $|\alpha'| > 0$, then $\alpha' = e\alpha''$ for some $e \in
vE^1$, and since $\Delta_v \le q_v - t_e t^*_e$, we have $\Delta_v t_{\alpha'} = 0$.
Symmetrically, $t^*_{\nu'} \Delta_w = 0$ if $\nu = \alpha\nu'$ and $|\nu'| > 0$. So if
$t_\mu \Delta_v t^*_\nu t_\alpha \Delta_w t^*_\beta \not= 0$, then $\nu = \alpha$, and
then $v = w$, and we obtain $t_\mu \Delta_v t^*_\nu t_\alpha \Delta_w t^*_\beta = t_\mu
\Delta_v q_v \Delta_v t^*_\beta$. Hence $t_\mu \Delta_v t^*_\nu t_\alpha \Delta_w
t^*_\beta = \delta_{\nu,\alpha} t_\mu \Delta_v t^*_\beta$ as required.

Since $E$ is row-finite, for any $v \in E^0$ we have $q_v \Delta_\mu = 0$ for all but
finitely many $\mu \in E^{<m}$. Since $\sum_v q_v$ converges strictly to the identity in
$\Mm( \Tt C^*(E) )$ and since $\{ \sum_{ \mu \in F } \Delta_{\mu} : \text{$F$ is a finite subset
of $E^{<m}$} \}$ is bounded, the series $\sum_{\mu \in E^{<m}} \Delta_\mu$ converges
strictly to a multiplier $\Phi_m$. This $\Phi_m$ is a projection because the preceding
paragraph shows that the $\Delta_\mu$ are mutually orthogonal projections. For $\mu \in
E^{<m}$ and $\alpha \in E^*$, we have
\[
\Delta_\mu t_\alpha
	= t_\mu \Delta_{r(\mu)} t^*_\mu t_\alpha
	= \begin{cases}
		t_\mu \Delta_{r(\mu)} t_{\alpha'} &\text{ if $\alpha = \mu\alpha'$}\\
		t_\mu \Delta_{r(\mu)} t^*_{\mu'} &\text{ if $\mu = \alpha\mu'$}\\
		0 &\text{ otherwise.}		
	\end{cases}
\]
Arguing as in the preceding paragraph, we see that the right-hand side is zero if $\alpha
= \mu\alpha'$ and $|\alpha'| > 0$. So
\[
\Delta_\mu t_\alpha
	= \begin{cases}
		t_\mu \Delta_{r(\mu)} t^*_{\mu'} &\text{ if $\mu = \alpha\mu'$}\\
		0 &\text{ otherwise.}		
	\end{cases}
\]
The formula given for $\Phi_m t_\alpha t^*_\beta \Phi_m$ now follows from the preceding
paragraph. Since $\Tt C^*(E) = \clsp\{t_\alpha t^*_\beta : \alpha,\beta \in E^*\}$, this
establishes the final assertion of the lemma as well.
\end{proof}

\begin{ntn}
Following \cite{WinterZacharias:AM10}, for each $m \in \NN$, we define $\kappa_m \in
M_m([0,\infty))$ to be the matrix such that for $i,j \le \lceil\frac{m}{2}\rceil$,
\begin{align*}
\kappa_m(i,j)
	= \kappa_m(m + 1 - i, j)
	&= \kappa_m(i, m + 1 - j)\\
	&= \kappa_m(m + 1 - i, m + 1 - j)
		= \frac{1}{\lceil\frac{m}{2}\rceil + 1}\min\{i,j\};
\end{align*}
so for $l \in \NN$,
\[
\kappa_{2l} = \frac{1}{l+1}
    \left(
    \begin{matrix}
        1 & 1 & \dots & 1 & 1 & 1 & 1 &\dots & 1 & 1 \\
        1 & 2 & \dots & 2 & 2 & 2 & 2 &\dots & 2 & 1 \\
        \vdots & \vdots & \ddots & \vdots &\vdots &\vdots &\vdots & \iddots & \vdots &\vdots  \\
        1 & 2 & \dots & l-1 & l-1 & l-1 & l-1 & \dots & 2 & 1 \\
        1 & 2 & \dots & l-1 & l & l & l-1 & \dots & 2 & 1 \\
        1 & 2 & \dots & l-1 & l & l & l-1 & \dots & 2 & 1 \\
        1 & 2 & \dots & l-1 & l-1 & l-1 & l-1 & \dots & 2 & 1 \\
        \vdots & \vdots & \iddots &\vdots &\vdots & \vdots &\vdots & \ddots &\vdots &\vdots  \\
        1 & 2 & \dots & 2 & 2 & 2 & 2 &\dots & 2 & 1 \\
        1 & 1 & \dots & 1 & 1 & 1 & 1 &\dots & 1 & 1 \\
    \end{matrix}
    \right),
\]
and
\[
\kappa_{2l + 1} = \frac{1}{l+2}
    \left(
    \begin{matrix}
        1 & 1 & \dots & 1 & 1 &  1 &\dots & 1 & 1 \\
        1 & 2 & \dots & 2 & 2 &  2 &\dots & 2 & 1 \\
        \vdots & \vdots & \ddots &\vdots &\vdots &\vdots & \iddots & \vdots &\vdots  \\
        1 & 2 & \dots & l & l & l & \dots & 2 & 1 \\
        1 & 2 & \dots & l & l+1 & l & \dots & 2 & 1 \\
        1 & 2 & \dots & l & l & l & \dots & 2 & 1 \\
        \vdots & \vdots & \iddots &\vdots &\vdots &\vdots & \ddots &\vdots &\vdots  \\
        1 & 2 & \dots & 2 & 2 & 2 &\dots & 2 & 1 \\
        1 & 1 & \dots & 1 & 1 & 1 &\dots & 1 & 1 \\
    \end{matrix}
    \right).
\]
As a notational convenience, we write $\kappa_{m}(i,j) = 0$ whenever $(i,j) \not\in \{1,
\dots, m\} \times \{1, \dots, m\}$.
\end{ntn}

Fix $m \in \NN$. Recall that Schur multiplication in $M_n(\CC)$ is entrywise
multiplication of matrices. For $l \le \lceil\frac{m}{2}\rceil$, let $P_{m,l} =
\sum^{m-l+1}_{i=l} \theta_{i,i}$ denote the projection onto $\lsp\{e_l, \dots,
e_{m-l+1}\}$. Then Schur multiplication by $\kappa_m$ is given by $\kappa_m * A =
\sum^{\lceil\frac{m}{2}\rceil}_{l=1} \frac{1}{\lceil\frac{m}{2}\rceil + 1} P_{m,l} A
P_{m,l}$. Hence Schur multiplication by $\kappa_m$ is completely positive with completely
bounded norm at most $\lceil\frac{m}{2}\rceil/(\lceil\frac{m}{2}\rceil + 1) < 1$.

\begin{lem}\label{lem:PmQm}
For each $m \in \NN$, there are completely positive contractions $P_m, Q_m : \Tt C^*(E)
\to \clsp\{t_\mu \Delta_{r(\mu)} t_\nu^* : \mu, \nu \in E^*\}$ such that, putting $l =
\lceil\frac{m}{2}\rceil$, we have
\begin{align*}
P_m(t_\mu t^*_\nu)
	&= \sum_{\substack{\tau \in r(\mu)E^* \\
			  m \le |\mu\tau|,|\nu\tau| < 2m}}
	    \kappa_m(|\mu\tau|-m, |\nu\tau|-m) t_{\mu\tau}\Delta_{r(\tau)} t^*_{\nu\tau}
    \quad\text{ and }\\
Q_m(t_\mu t^*_\nu)
	&= \sum_{\substack{\tau \in r(\mu)E^* \\
			  m+l \le |\mu\tau|,|\nu\tau| < 2m+l}}
		\kappa_m(|\mu\tau|-(m+l), |\nu\tau|-(m+l)) t_{\mu\tau}\Delta_{r(\tau)} t^*_{\nu\tau}.
\end{align*}
\end{lem}
\begin{proof}
We argue the case for $P_m$; the situation for $Q_m$ is similar. The multiplier
$\Phi_{2m} - \Phi_{m}$ obtained from Lemma~\ref{lem:compacts}
is a projection, and so compression by this element determines a
completely positive contraction. Let $M$ be the block operator matrix $M = \sum_{p,q}
\kappa_m(p - m, q - m) 1_{E^p \times E^q}$, where $1_{X,Y}$ denotes the matrix in
$M_{X,Y}(\CC)$ whose entries are all $1$. Since Schur multiplication by $\kappa_m$ is a
completely positive complete contraction, Schur multiplication $a \mapsto M*a$ by $M$ is
a completely positive contraction. So $P_m : a \mapsto M * (\Phi_{2m} - \Phi_{m}) a
(\Phi_{2m} - \Phi_{m})$ is a completely positive contraction satisfying the desired
formula.
\end{proof}

For what follows, we need to recall a construction from Section~4 of \cite{KribsSolel:JAMS07}.
If $E$ is a directed graph and $m$ is a positive integer, we define $E(m)$ to be the
directed graph with
\begin{align*}
E(m)^0 &= E^{<m},&
E(m)^1 &= \{(e, \mu) : e \in E^1, \mu \in E^{<m}, r(e) = s(\mu)\},\\
r(e,\mu) &= \mu,&
s(e,\mu) &= \begin{cases}
		e\mu &\text{ if $|\mu| < m-1$}\\
		s(e) &\text{ if $|\mu| = m-1$.}
	\end{cases}
\end{align*}
For $\mu \in E^*$, we write $[\mu]_m$ for the unique element of $E^{<m}$ such that $\mu =
[\mu]_m \mu'$ with $|\mu'| \in m\NN$. There is an injection $i_m : E^* \to E(m)^*$ given
by
\[
i_m(v) = v\quad\text{ for $v \in E^0$},
\]
and
\[
i_m(\mu)
	= (\mu_1, \mu_2\mu_3 \dots \mu_{|\mu|})
		(\mu_2, \mu_3 \dots \mu_{|\mu|})
		\dots
		(\mu_{|\mu|}, r(\mu))\quad\text{ if $1 \le |\mu| \le m$,}
\]
and recursively by $i_m(\mu) = i_m([\mu]_m) i_m(\mu') i_m(\mu'')$ when $\mu$ is
factorized as $\mu = [\mu]_m\mu'\mu''$ with $|\mu'| = m$. We have $s_{E(m)}(i_m(\mu)) =
[\mu]_m$ and $r_{E(m)}(i_m(\mu)) = r_E(\mu)$.

In the following, given $p < q \in \NN$, we shall write $E^{[p,q)}$ for the set $\{\mu
\in E^* : p \le |\mu| < q\}$.
		
\begin{lem}\label{lem:Lambda homomorphism}
Let $E$ be a row-finite directed graph. For $p,m \in \NN$ there is a homomorphism
$\Lambda_p^{p+m} : \bigoplus_{v \in E^0} \Kk_{E^{[p, p+m)} v} \to \Tt C^*(E(m))$ such
that
\[
\Lambda_p^{p+m}(\theta_{\mu,\nu}) = t_{i_m(\mu)} t^*_{i_m(\nu)}
\]
for all $\mu,\nu \in E^{[p, p+m)}$ with $r(\mu) = r(\nu)$.
\end{lem}
\begin{proof}
We clearly have $(t_{i_m(\mu)} t^*_{i_m(\nu)})^* = t_{i_m(\nu)} t^*_{i_m(\mu)}$. We show
that
\[
t_{i_m(\mu)} t^*_{i_m(\nu)} t_{i_m(\alpha)} t^*_{i_m(\beta)}
	= \delta_{\nu, \alpha} t_{i_m(\mu)} t^*_{i_m(\beta)}.
\]
First suppose $\nu = \alpha$. Then $i_m(\nu) = i_n(\alpha)$. Hence~(TCK1) gives
$t_{i_m(\mu)} t^*_{i_m(\nu)} t_{i_m(\alpha)} t^*_{i_m(\beta)} = t_{i_m(\mu)} q_{r(\nu)}
t_{i_m(\beta)}^*$. Since $r(\mu) = r(\nu)$, we have $r_{E(m)}(i(\mu)) = r(\nu)$, and so
applying~(TCK1) again gives $t_{i_m(\mu)} q_{r(\nu)} t_{i_m(\beta)}^* = t_{i_m(\mu)}
t_{i_m(\beta)}^*$ as required. Now suppose that $\nu \not= \alpha$. Assume without loss
of generality that $|\nu| \ge |\alpha|$. We consider two cases. First suppose that
$\nu\not= \alpha\nu'$ for any $\nu'$. Then there exists $i < |\alpha|$ such that $\nu_i
\not= \alpha_i$. Since $i_m(\nu)_i = (\nu_i, \tau)$ and $i_m(\alpha)_i = (\alpha_i,
\rho)$ for some $\tau,\rho \in E^{<m}$, we have $i_m(\nu)_i \not= i_m(\alpha)_i$. In
particular $i_m(\nu)$ does not have the form $i_m(\alpha)\beta$ for $\beta \in E(m)^*$,
and so~(TCK1) in $\Tt C^*(E(m))$ gives $t_{i_m(\nu)}^* t_{i_m(\alpha)} = 0$. Now suppose
that $\nu = \alpha\nu'$. Then $|\nu'| > 0$, and since $\alpha,\nu \in E^{[p, p+m)}$, we
have $|\nu'| < m$. Thus $|\nu| \not\equiv |\alpha| \mod m$, and in particular
$s(i_m(\nu)) = [\nu]_m \not= [\alpha]_m = s(i_m(\alpha))$. So~(TCK1) implies that
$t_{i_m(\nu)}^* t_{i_m(\alpha)} = 0$.

Hence $\{t_{i_m(\mu)} t^*_{i_m(\nu)} : \mu,\nu \in E^{[p, m+p)} v\}$ is a family of
matrix units for each $v$, and it follows from the universal properties of the
$\Kk_{E^{<m}v}$ that there is a homomorphism $\Lambda_p^{p+m}$ as claimed.
\end{proof}

The next lemma and its proof are due to James Rout, and will appear in his PhD thesis. We
thank James for providing us with the details.

\begin{lem}[Rout]\label{lem:rout}
If $E$ is a row-finite directed graph and $m \ge 0$, then there is an injective
homomorphism $\iota_m : \Tt C^*(E) \to \Tt C^*(E(m))$ such that
\begin{equation}\label{eq:JRinclusion}
\iota_m(q_v) = \sum_{\mu \in vE^{<m}} q^m_\mu\quad\text{ and }\quad
\iota_m(t_e) = \sum_{(e,\mu) \in E(m)^1} t^m_{(e,\mu)}.
\end{equation}
The map $\iota_m$ descends to an injective homomorphism $\widetilde{\iota}_m : C^*(E) \to
C^*(E(m))$.
\end{lem}
\begin{proof}
Routine calculations show that the $\iota_m(q_v)$ and $\iota_m(t_e)$ form a
Toeplitz-Cuntz-Krieger $E$-family, and so induce a homomorphism $\iota_m : \Tt C^*(E) \to
\Tt C^*(E(m))$ satisfying the desired formula. For $v \in E^0$ we have
\begin{equation}\label{eq:CKcompare}
\iota_m\Big(q_v - \sum_{e \in vE^1} t_e t^*_e\Big)
    = \sum_{\mu \in v E^{<m}} \Big(q^m_\mu - \sum_{s_{E(m)}(e, \nu) = \mu} t^m_{(e,\nu)} (t^m_{(e,\nu)})^*\Big).
\end{equation}
This is nonzero because each $q^m_\mu - \sum_{\lambda \in \mu E(m)^1} t^m_\lambda
(t^m_\lambda)^*$ is nonzero in $\Tt C^*(E(m))$. Hence
\cite[Theorem~2.1]{FowlerRaeburn:IUMJ99} implies that $\iota_m$ is injective.

The quotient map $\pi_m : \Tt C^*(E(m)) \to C^*(E(m))$ carries each term in the
right-hand side of~\eqref{eq:CKcompare} to zero, and so $\pi_m \circ \iota_m$ descends to
a homomorphism $\tilde{\iota}_m : C^*(E) \to C^*(E(m))$. The $\tilde{\iota}_m(p_v)$ are
all nonzero because the vertex projections $\pi_m(q^m_\mu)$ in $C^*(E(m))$ are nonzero.
The homomorphism $\tilde{\iota}_m$ is clearly equivariant for the gauge actions on
$C^*(E)$ and $C^*(E(m))$, and so the gauge-invariant uniqueness theorem
\cite[Theorem~3.1]{BatesPaskEtAl:NYJM00} implies that $\tilde{\iota}_m$ is injective.
\end{proof}

\begin{lem}\label{lem:equiv-exprs}
Let $E$ be a row-finite directed graph and let $m \ge 1$. Let $q, t$ be the universal
Toeplitz-Cuntz-Krieger family in $\Tt C^*(E)$, and let $Q,T$ be the universal
Toeplitz-Cuntz-Krieger family in $\Tt C^*(E(m))$. For $\mu \in E^*$, we have
\[
	T_{i_m(\mu)} = \iota_m(t_\mu) Q_{r(\mu)} = Q_{[\mu]_m} \iota_m(t_\mu).
\]
\end{lem}
\begin{proof}
First suppose that $l := |\mu| < m$. Then
\begin{align*}
\iota_m(t_\mu) Q_{r(\mu)}
	&= \iota_m(t_{\mu_1}) \dots \iota_m(t_{\mu_{l-1}}) \iota_m(t_{\mu_l}) Q_{r(\mu)}\\
	&= \Big(\sum_{(\mu_1, \nu) \in E(m)^1} T_{(\mu_1, \nu)}\Big)
		\dots
		\Big(\sum_{(\mu_{l-1}, \nu) \in E(m)^1} T_{(\mu_{l-1}, \nu)}\Big)
		\Big(\sum_{(\mu_l, \nu) \in E(m)^1} T_{(\mu_l, \nu)}\Big) Q_{r(\mu)}\\
	&= \Big(\sum_{(\mu_1, \nu) \in E(m)^1} T_{(\mu_1, \nu)}\Big)
		\dots
		\Big(\sum_{(\mu_{l-1}, \nu) \in E(m)^1} T_{(\mu_{l-1}, \nu)}\Big)
		T_{(\mu_l, r(\mu))}\\
	&= \Big(\sum_{(\mu_1, \nu) \in E(m)^1} T_{(\mu_1, \nu)}\Big)
		\dots
		\Big(\sum_{(\mu_{l-1}, \nu) \in E(m)^1} T_{(\mu_{l-1}, \nu)}\Big)
		Q_{\mu_l} T_{(\mu_l, r(\mu))} \\
	&\;\vdots \\
	&= T_{(\mu_1, \mu_2\dots\mu_l)} \dots T_{(\mu_l, r(\mu))} \\
	&= T_{i_m(\mu)},
\end{align*}
and
\begin{align*}
Q_{[\mu]_m} \iota_m(t_\mu)
	&= Q_{\mu} \iota_m(t_\mu)\\
	&= Q_{\mu} \Big(\sum_{(\mu_1, \nu) \in E(m)^1} T_{(\mu_1, \nu)}\Big)
		\dots
		\Big(\sum_{(\mu_l, \nu) \in E(m)^1} T_{(\mu_l, \nu)}\Big)\\
    &= T_{(\mu_1, \mu_2 \dots \mu_l)}
    	\Big(\sum_{(\mu_2, \nu) \in E(m)^1} T_{(\mu_2, \nu)}\Big)
   		\dots
   		\Big(\sum_{(\mu_l, \nu) \in E(m)^1} T_{(\mu_l, \nu)}\Big)\\
   	&\;\vdots\\
   	&= T_{(\mu_1, \mu_2 \dots \mu_l)}
   	   \dots
   	   T_{(\mu_l, r(\mu))}\\
   	&= T_{i_m(\mu)}.	
\end{align*}
Now suppose that $|\mu| = m$. Then, using the above calculations, we have
\begin{align*}
\iota_m(t_\mu) Q_{r(\mu)}
	&= \Big(\sum_{(\mu_1, \nu) \in E(m)^1} T_{(\mu_1, \nu)}\Big) Q_{\mu_2 \dots\mu_m} 	
		T_{(\mu_2, \mu_3\dots\mu_l)} \dots T_{(\mu_m, r(\mu))}\\
	&= Q_{s(\mu_1, \mu_2\dots\mu_m)} T_{(\mu_1, \mu_2\dots\mu_m)}
		T_{(\mu_2, \mu_3\dots\mu_l)} \dots T_{(\mu_m, r(\mu))}
	= T_{i_m(\mu)},
\end{align*}
and
\begin{align*}
Q_{[\mu]_m} \iota_m(t_\mu)
	&= Q_{s(\mu)} \Big(\sum_{(\mu_1, \nu) \in E(m)^1} T_{(\mu_1, \nu)}\Big)
		\dots \Big(\sum_{(\mu_m, \nu) \in E(m)^1} T_{(\mu_m, \nu)}\Big) \\
	&= \sum_{\nu \in r(\mu_1)E^{m-1}} \Big(T_{(\mu_1, \nu)}
		\Big(\sum_{(\mu_2, \eta) \in E(m)^1} Q_\nu T_{(\mu_2, \eta)}\Big) \dots
		\Big(\sum_{(\mu_m, \eta) \in E(m)^1} T_{(\mu_m, \eta)}\Big)\Big) \\
	&= \sum_{\nu \in r(\mu_2)E^{m-2}} \Big(T_{(\mu_1, \mu_2\nu)}
		T_{(\mu_2, \nu)}\Big(\sum_{(\mu_3,\eta) \in E(m)^1} T_{(\mu_3, \eta)}\Big) \dots
		\Big(\sum_{(\mu_m, \nu) \in E(m)^1} T_{(\mu_m, \nu)}\Big)\Big) \\
	&\;\vdots\\
	&= T_{(\mu_1, \mu_2\dots\mu_m)} T_{(\mu_2, \mu_3\dots\mu_l)} \dots T_{(\mu_m,
	r(\mu))}
	= T_{i_m(\mu)}.
\end{align*}
Now a straightforward induction on $|\mu|$ using the case $|\mu| \le m$ as a base case
establishes the result.
\end{proof}

\begin{prp}\label{prp:approximates}
Let $E$ be a row-finite directed graph with no sinks. For each integer $m \ge 1$, let $q_m : \Tt
C^*(E(m)) \to C^*(E(m))$ be the quotient map. Then for $\mu,\nu \in E^*$ with $r(\mu) =
r(\nu)$, we have
\[
\big\|q_m\big(\Lambda_{m}^{2m}(P_m(t_\mu t^*_\nu)) +
	\Lambda^{\lceil\frac{5m}{2}\rceil}_{\lceil\frac{3m}{2}\rceil}
		(Q_m(t_\mu t^*_\nu))\big) - q_m(\iota_{m}(t_\mu t^*_\nu))\big\| \to 0
\]
as $m \to \infty$.
\end{prp}
\begin{proof}
Identify each $\Kk_{E^{<p}}$ with $\clsp\{s_\mu \Delta_{r(\mu)} s^*_\nu : \mu,\nu \in
E^{<p}\}$ as in Lemma~\ref{lem:compacts}. Fix $\mu,\nu$ with $r(\mu) = r(\nu) = v$.
Combining the formula for $P_m(t_\mu t^*_\nu)$ from Lemma~\ref{lem:PmQm} with our
convention that $\kappa_m(i,j) = 0$ if $(i,j) \not\in \{1, \dots, m\} \times \{1, \dots,
m\}$, we have
\[
P_m(t_\mu t^*_\nu)
	= \sum_{\tau \in r(\mu)E^*}
		    \kappa_m(|\mu\tau|-m, |\nu\tau|-m) t_{\mu\tau}\Delta_{r(\tau)} t^*_{\nu\tau}.
\]
Applying the definition of $\Lambda_{m}^{2m}$ and then Lemma~\ref{lem:equiv-exprs}, we
obtain
\begin{align*}
\Lambda_{m}^{2m}(P_m(t_\mu t^*_\nu))
	&=  \sum_{\tau \in r(\mu)E^*}
	    \kappa_m(|\mu\tau|-m, |\nu\tau|-m)
	     t_{i_{m}(\mu\tau)} t^*_{i_{m}(\nu\tau)} \\
	&= \sum^{m-1}_{i=0} \sum_{j=1}^\infty
		\kappa_m(|\mu| + i + m(j-1), |\nu| + i + m(j-1))  \\
		&\qquad\Big(\sum_{\alpha \in r(\mu)E^i} \iota_{m}(t_{\mu\alpha})
     			\Big(\sum_{\rho \in r(\alpha)E^{mj}} \iota_{m}(t_\rho t^*_\rho)\Big)
				Q_{\alpha} \iota_{m}(t^*_{\nu\alpha})\Big).
\end{align*}
Suppose that $m > |\mu|, |\nu|$. Then for each $i < m$ there exists a unique $j(i) \in
\ZZ$ such that $0 \le |\mu| + i + m(j(i)-1) < m$. So applying $q_m$ to both sides of the
preceding calculation yields
\begin{align*}
q_m&(\Lambda_{m}^{2m}(P_m(t_\mu t^*_\nu))) \\
	&= \sum^{m-1}_{i=0}
	\kappa_m(|\mu| + i + m(j(i)-1), |\nu| + i + m(j(i)-1)) \\
	&\hskip1cm\Big(\sum_{\alpha \in r(\mu)E^i} \tilde{\iota}_{m}(s_{\mu\alpha})
     			\tilde{\iota}_{m}\Big(\sum_{\tau \in r(\alpha)E^{mj(i)}} s_\tau s^*_\tau\Big)
				\tilde{\iota}_{m}(s^*_{\alpha}) P_{\alpha} \tilde{\iota}_{m}(s^*_{\nu})\Big)\\
	&=  \sum^{m-1}_{i=0}
		\big(\kappa_m(|\mu| + i + m(j(i)-1), |\nu| + i + m(j(i)-1)) \\
		&\hskip1cm\Big(\sum_{\alpha \in r(\mu)E^i} \tilde{\iota}_{m}(s_{\mu\alpha})
					P_{r(\alpha)} \tilde{\iota}_{m}(s^*_{\nu\alpha})\Big),
\end{align*}
where the last equality comes from the Cuntz-Krieger relation and
Lemma~\ref{lem:equiv-exprs}.

Similar reasoning gives
\begin{align*}
q_m(\Lambda^{\lceil\frac{5m}{2}\rceil}_{\lceil\frac{3m}{2}\rceil}
	(Q_m(t_\mu t^*_\nu)))
	&=  \sum^{m-1}_{i=0}
		\kappa_m(|\mu| + i + \lceil m(k(i)-3/2)\rceil, |\nu| + i + \lceil m(k(i)-3/2)\rceil) \\
		&\hskip4cm\Big(\sum_{\alpha \in r(\mu)E^i} \tilde{\iota}_{m}(s_{\mu\alpha})
					P_{r(\alpha)} \tilde{\iota}_{m}(s^*_{\nu\alpha})\Big),
\end{align*}
where each $k(i)$ is the unique integer such that $0 \le |\mu| + i + \lceil
m(k(i)-3/2)\rceil < m$. Since $|\mu| + i + m(j(i)-1)$ and $|\mu| + i + \lceil
m(k(i)-3/2)\rceil$ differ by an odd multiple of $\lceil\frac{m}{2}\rceil$, the sum
\begin{align*}
K_{m,i} &:= \kappa_m\Big(|\mu| + i + m(j(i)-1), |\nu| + i + m(j(i)-1)\Big) \\
	&\qquad{}+ \kappa_m\Big(|\mu| + i + \lceil m(k(i)-3/2)\rceil, |\nu| + i + \lceil m(k(i)-3/2)\rceil\Big)
\end{align*}
is either $\frac{m - \big||\mu| - |\nu|\big|}{m+1}$ or $\frac{m + 1 - \big||\mu| -
|\nu|\big|}{m+1}$ for each $i \le m$. In particular
\begin{equation}\label{eq:K-estimate}
	0 < 1 - K_{m,i} < \frac{\big||\mu| - |\nu|\big|}{m+1}
\end{equation}
for all $m$ and all $i \le m$.

So for large $m$, we have
\begin{align*}
q_m\big(\Lambda_{m}^{2m}(P_m(t_\mu t^*_\nu))\big)
	&{}+ \Lambda^{\lceil\frac{5m}{2}\rceil}_{\lceil\frac{3m}{2}\rceil}
		(Q_m(t_\mu t^*_\nu))\\
	&=\tilde{\iota}_{m}(s_\mu)
		\Big(\sum_{\alpha \in r(\mu)E^{<m}}
		K_{m,|\alpha|} \tilde{\iota}_{m}(s_{\alpha}) P_{r(\alpha)} \tilde{\iota}_{m}(s^*_{\alpha})\Big)\tilde{\iota}_{m}(s^*_\nu)\\
	&= \tilde{\iota}_{m}(s_\mu)
		\Big(\sum_{\alpha \in r(\mu)E^{<m}} K_{m,|\alpha|} S_{i_{m}(\alpha)} S^*_{i_{m}(\alpha)}\Big)\tilde{\iota}_{m}(s^*_\nu),
\end{align*}
where we have once again used Lemma~\ref{lem:equiv-exprs} in the last line. For $0 <
|\alpha| < m$, we have $s_{E(m)}^{-1}(\alpha) = \{i_{m}(\alpha)\}$, and so
$S_{i_{m}(\alpha)} S^*_{i_{m}(\alpha)} = P_\alpha$. So the final line of the preceding
display becomes
\[
\tilde{\iota}_{m}(s_\mu)
		\Big(K_{m,0} S_{r(\mu)} S^*_{r(\mu)} + \sum_{\alpha \in r(\mu)E^{<m} \setminus \{r(\mu)\}} K_{m,|\alpha|} P_\alpha\Big)\tilde{\iota}_{m}(s^*_\nu).
\]
Since $\tilde{\iota}(P_{r(\mu)}) = \sum_{\alpha \in r(\mu)E^{<m}} P_\alpha$, the
estimate~\eqref{eq:K-estimate} shows that
\[
\Big\|\Big(K_{m,0} P_{r(\mu)}
	+ \sum_{\alpha \in r(\mu)E^{<m} \setminus \{r(\mu)\}} K_{m,|\alpha|}\Big)
	- \tilde{\iota}(P_{r(\mu)})\Big\|
	\le \frac{\big||\mu| - |\nu|\big|}{m+1},
\]
and so
\[
\big\|q_m\big(\Lambda_{m}^{2m}(P_m(t_\mu t^*_\nu)) +
	\Lambda^{\lceil\frac{5m}{2}\rceil}_{\lceil\frac{3m}{2}\rceil}
		(Q_m(t_\mu t^*_\nu))\big) - \tilde{\iota}_{m}(s_\mu s^*_\nu))\big\|
	\le \frac{\big||\mu| - |\nu|\big|}{m+1}
\]
as well.
\end{proof}

\begin{dfn}
Suppose that $(\beta_m)^\infty_{m=1}$ is a sequence of $C^{*}$-homomorphisms $\beta_m : A
\to B_m$, and let $\mathcal{C}$ be a class of $C^*$-algebras. We say that a sequence
$(F_m, \psi_m, \varphi_m)$ is an \emph{asymptotic order-$n$ factorization of the sequence
$(\beta_m)$ through elements of $\mathcal{C}$} if each $F_m$ is a direct sum $F_m =
\bigoplus^n_{i=0} F_m^{(i)}$ of $C^*$-algebras $F_m^{(i)} \in \mathcal{C}$, each $\psi_m$
is a completely positive contraction from $A$ to $F_m$, each $\varphi_m$ is a map from
$F_m$ to $B_m$ such that $\varphi_m|_{F_m^{(i)}}$ is an order-zero completely positive
contraction for each $i \le n$, and $\|\varphi_m \circ \psi_m(a) - \beta_m(a)\| \to 0$
for each $a \in A$. We say that $(F_m, \psi_m, \varphi_m)$ is an asymptotic order-$n$
factorization of a fixed $C^*$-homomorphism $\beta : A \to B$ if it is an asymptotic
order-$n$ factorization of the constant sequence $(\beta, \beta, \beta, \dots)$.
\end{dfn}

\begin{lem}\label{lem:order-zero-AF-FD}
For each $m \in \NN$, let $\beta_m : A \to B_m$ be a homomorphism of separable $C^*$-algebras. If the sequence
$(\beta_m)$ has an asymptotic order-$n$ factorization through AF algebras, then it has an
asymptotic order-$n$ factorization through finite-dimensional $C^*$-algebras.
\end{lem}
\begin{proof}
Choose an asymptotic order-$n$ factorization $(F_m, \psi_m, \varphi_m)$ of $(\beta_m)$
through AF algebras. Choose a dense sequence $(a_j)$ in $A$. By passing to a subsequence
in $m$, we may assume that $\|\varphi_{m} \circ \psi_{m}(a_j) - \beta_{m}(a_j)\| < 1/2m$
whenever $j \le m$. For each $m$, choose finite dimensional $\widetilde{F}_m^{(i)}
\subseteq F_m^{(i)}$ with $d\big(\psi_m(a_j), \bigoplus_i \widetilde{F}_m^{(i)}\big) <
1/2m$ for all $j \le m$. Since $\widetilde{F}_m := \bigoplus_i \widetilde{F}_m^{(i)}$ is
finite dimensional there is a completely positive contraction $\gamma_m : F_{m} \to
\widetilde{F}_m$ fixing $\widetilde{F}_m$ pointwise. So for $j \le m$, we have
$\|\gamma_m(\psi_m(a_j)) - \psi_m(a_j)\| < 1/2m$, and hence
\begin{align*}
\|\varphi_{m} \circ \gamma_m \circ \psi_{m}(a_j) &{}- \beta_m(a_j)\| \\
	&\leq \|\varphi_{m}(\gamma_m(\psi_{m}(a_j)) - \psi_m(a_j)) \|
	   + \|\varphi_{m}(\psi_m(a_j)) - \beta_m(a_j)\|
	< 1/m.
\end{align*}
Now each $\widetilde{\psi}_m := \gamma \circ \psi_{m} : A \to \widetilde{F}_m$ is a
completely positive contraction, each $\varphi_m$ restricts to a completely positive
order-zero contraction from $\widetilde{F}_m^{(i)}$ to $B_m$, and
$\|\widetilde{\varphi}_m \circ \widetilde{\psi}_m(a) - \beta_m(a)\| \to 0$ for all $a \in
A$ because the $a_i$ are dense.
\end{proof}

\begin{cor}\label{cor:order-zero-approximation}
Let $E$ be a row-finite directed graph with no sinks. For each $m$, let $\tilde{\iota}_{m} : C^*(E) \to
C^*(E(m))$ be the homomorphism induced by the homomorphism~\eqref{eq:JRinclusion}. Then
the sequence $(\tilde{\iota}_m)$ has an asymptotic order-$1$ factorization through
finite-dimensional $C^*$-algebras.
\end{cor}
\begin{proof}
Since $C^{*} (E)$ is nuclear, there is a norm-1 completely positive splitting $\sigma :
C^*(E) \to \Tt C^*(E)$ for the quotient map \cite[Theorem~3.10]{ChoiEffros:Ann76}.
Identify each $\Kk_{E^{<p}}$ with $\clsp\{s_\mu \Delta_{r(\mu)} s^*_\nu : \mu,\nu \in
E^{<p}\}$ as in Lemma~\ref{lem:compacts}. For each $m$, define $\psi_m : C^*(E) \to
\Kk_{E^{[m,2m)}} \oplus \Kk_{E^{[\lceil\frac{3m}{2}\rceil, \lceil\frac{5m}{2}\rceil)}}$
by $\psi_m(a) = P_m(\sigma(a)) \oplus Q_m(\sigma(a))$, and define $\varphi_m :
\Kk_{E^{[m,2m)}} \oplus \Kk_{E^{[\lceil\frac{3m}{2}\rceil, \lceil\frac{5m}{2}\rceil)}}
\to C^*(E(m))$ by $\varphi_m((a,b)) = q_m(\Lambda_{m}^{2m}(a) +
\Lambda^{\lceil\frac{5m}{2}\rceil}_{\lceil\frac{3m}{2}\rceil}(b))$.
Proposition~\ref{prp:approximates} shows that $\|\varphi_m \circ \psi_m(s_\mu s^*_\nu) -
\tilde{\iota}_m(s_\mu s^*_\nu)\| \to 0$ for all $\mu,\nu$, and then an
$\varepsilon/3$-argument proves that $\|\varphi_m \circ \psi_m(a) - \tilde{\iota}_m(a)\|
\to 0$ for each $a \in C^*(E)$. Lemma~\ref{lem:Lambda homomorphism} shows that each
$\varphi_m$ restricts to a homomorphism, and hence a completely positive contraction of
order zero, on each of $\Kk_{E^{[m,2m)}}$ and $\Kk_{E^{[\lceil\frac{3m}{2}\rceil,
\lceil\frac{5m}{2}\rceil)}}$.  The corollary now follows from
Lemma~\ref{lem:order-zero-AF-FD}.
\end{proof}

\section{Reinclusion in $K$-theory}\label{sec:reinclusion}

In this section we describe, for each row-finite graph $E$ with no sinks, an inclusion
\[
C^*(E(m)) \hookrightarrow C^*(E) \otimes \Kk
\]
which induces the multiplication-by-$m$ map on $K$-theory when composed with Rout's
inclusion $\tilde{\iota}_m : C^*(E) \to C^*(E(m))$.

\begin{prp}\label{prp:E(m)->ExK}
Let $E$ be a row-finite directed graph with no sinks. For each $m \in \NN$ there is an
injective homomorphism $j_m : C^*(E(m)) \to C^*(E) \otimes \Kk_{E^{<m}}$ such that
\begin{equation}\label{eq:jsubmdef}
j_m(p_{\mu}) = p_{r(\mu)} \otimes \theta_{\mu,\mu}
	\qquad\text{and}\qquad
j_m(s_{(e,\mu)})
	= \begin{cases}
		s_{e\mu} \otimes \theta_{s(e), \mu} &\text{ if $|\mu| = m-1$}\\
		p_{r(\mu)} \otimes \theta_{e\mu, \mu} &\text{ otherwise.}	
	\end{cases}
\end{equation}
\end{prp}
\begin{proof}
We check the Cuntz-Krieger relations. For~(CK1), we calculate:
\[
j_m(s_{(e,\mu)})^*j_m(s_{(e,\mu)})
	= \left.\begin{cases}
		s_{e\mu}^* s_{e\mu} \otimes \theta_{\mu, \mu}
			&\text{ if $|\mu| = m-1$}\\
		p_{r(\mu)} \otimes \theta_{\mu, \mu}
			&\text{ otherwise}
	\end{cases}\right\}
	= j_m(p_\mu) = j_m(p_{r(e,\mu)}).
\]
For~(CK2), we fix $\mu \in E^{<m} = E(m)^0$, and consider two cases. First suppose that
$\mu = e \mu'$ for some $e \in E^1$ and $\mu' \in E^*$. Then $\mu E(m)^1= \{(e, \mu')\}$
and $|\mu'| < m-1$, so we have
\[
\sum_{(f,\nu) \in \mu E(m)^1} j_m(s_{f,\nu}) j_m(s_{f,\nu})^*
	= j_m(s_{(e, \mu')}) j_m(s_{(e,\mu')})^*
	= p_{r(\mu')} \otimes \theta_{e\mu', e\mu'}
	= j_m(p_{\mu}).
\]
Now suppose that $|\mu| = 0$ so that $\mu = v \in E^0$. We have $vE(m)^1 = \{(e,\nu) :
s(e) = v\text{ and } \nu \in vE^{m-1}\}$. Since $E$ has no sinks, each $p_v =
\sum_{\lambda \in vE^m} s_\lambda s^*_\lambda$ in $C^*(E)$, and so
\begin{align*}
\sum_{(f,\nu) \in v E(m)^1} j_m(s_{f,\nu}) j_m(s_{f,\nu})^*
	&= \sum_{f\nu \in vE^m} s_{f\nu} s_{f\nu}^* \otimes \theta_{s(f), s(f)}\\
	&= \Big(\sum_{f\nu \in vE^m} s_{f\nu} s_{f\nu}^*\Big) \otimes \theta_{v, v}
	= p_{v} \otimes \theta_{v,v}
	= j_m(p_v).
\end{align*}
So the $j_m(p_\mu)$ and the $j_m(s_e)$ form a Cuntz-Krieger $E(m)$-family in $C^*(E)
\otimes \Kk_{E^{<m}}$ and the universal property of $C^*(E(m))$ ensures that $j_m$
extends to a homomorphism of $C^*$-algebras. The $j_m(p_\mu)$ are clearly all nonzero.
Let $\gamma$ denote the gauge action on $C^*(E)$, and let $\beta$ denote the action of
$\TT$ on $\Kk_{E^{<m}}$ determined by $\beta_z(\theta_{\mu,\nu}) = z^{|\mu| - |\nu|}
\theta_{\mu,\nu}$. It is routine to check that $(\gamma \otimes \beta)_z(j_m(p_\mu)) =
j_m(p_\mu)$ for all $\mu$ and that $(\gamma \otimes \beta)_z(j_m(s_{(e,\mu)})) =
zj_m(s_{(e,\mu)})$ for all $(e,\mu)$. So the gauge-invariant uniqueness theorem
\cite[Theorem~2.1]{BatesPaskEtAl:NYJM00} (see also
\cite[Theorem~2.3]{HuefRaeburn:ETDS97}) implies that $j_m$ is injective.
\end{proof}

Recall that there is an inclusion $\tilde{\iota}_m : C^*(E) \to C^*(E(m))$ satisfying the
formula~\eqref{eq:JRinclusion}, and so we have an inclusion $j_m \circ \tilde{\iota}_m :
C^*(E) \to C^*(E) \otimes \Kk_{E^{<m}}$.

\begin{lem}\label{lem:times m}
Let $E$ be a row-finite directed graph with no sinks. Identify $K_*(C^*(E) \otimes
\Kk_{E^{<m}} )$ with $K_*(C^*(E))$. Then the induced map $(j_m \circ \tilde{\iota}_m)_* :
K_*(C^*(E)) \to K_*(C^*(E))$ is multiplication by $m$.
\end{lem}
\begin{proof}
We first recall from \cite[page~439]{PaskRaeburn:PRIMS96} (Theorem~4.2.4 and the
discussion immediately preceding it) the computation of the $K$-theory of $C^*(E)$. We
can identify $K_0(C^*(E)^\gamma)$ with $\varinjlim (\ZZ E^0, A^t)$ where $A$ is the $E^0
\times E^0$ matrix $A(v,w) = |vE^1w|$. We then have $K_1(C^*(E)) \cong \ker \phi$ and
$K_0(C^*(E)) \cong \coker \phi$ where $\phi : \varinjlim (\ZZ E^0, A^t) \to \varinjlim
(\ZZ E^0, A^t)$ is the homomorphism induced by $1 - A^t$. Identify $\varinjlim (\ZZ E^0,
A^t)$ with tail-equivalence classes of sequences $(a_j) \in (\ZZ E^0)^\infty$ such that
$a_{j+1} = A^t a_j$ for large $j$, and define $\omega : \ZZ E^\infty \to \varinjlim (\ZZ
E^0, A^t)$ by $\omega(x) = (x, A^t x, (A^t)^2 x, \dots)$. Then Theorem~4.2.4 of
\cite{PaskRaeburn:PRIMS96} says that $\omega$ induces isomorphisms of $\ker(1 - A^t)
\cong \ker(\phi)$ and $\coker(1 - A^t) \cong \coker\phi$.

It is routine to check that $j_m \circ \tilde{\iota}_m(p_v)$ carries $C^*(E)^\gamma$ into
$C^*(E)^\gamma \otimes \Kk_{E^{<m}}$. We consider the effect of this map on
$K_0(C^*(E)^\gamma)$: fix $v \in E^0$ and calculate
\begin{align*}
j_m \circ \tilde{\iota}_m(p_v)
	&= \sum_{\mu \in vE^{<m}} j_m(p_\mu)
	= \sum_{\mu \in vE^{<m}} p_{r(\mu)} \otimes \theta_{\mu,\mu};
\end{align*}
so composing with the isomorphism $K_0(C^*(E)^\gamma) \cong K_0(C^*(E)^\gamma \otimes
\Kk_{E^{<m}} )$, the induced map $(j_m \circ \tilde{\iota}_m)_* : K_0(C^*(E)^\gamma) \to
K_0(C^*(E)^\gamma)$ is given by
\[
[p_v]_0
	\mapsto \sum_{\mu \in vE^{<m}} [p_{r(\mu)}]_0
	= \sum^{m-1}_{j=0} \sum_{w \in E^0} A^t(w,v) [p_w]_0
	= \sum^{m-1}_{j=0} (A^t)^j \cdot [p_v]_0.
\]
Suppose that $x \in \ker(1 - A^t)$. Then $\omega(x) = (x,x,x,\dots)$ and $A^t x = x$, and
so we have
\[
(j_m \circ \tilde{\iota}_m)_*(\omega(x))
	= \sum^{m-1}_{j=0} ((A^t)^jx, (A^t)^jx, \dots)
	= m \cdot (x, x, \dots)
	= m \cdot \omega(x),
\]
and we deduce that $(j_m \circ \tilde{\iota}_m)_*$ is multiplication by $m$ on
$K_1(C^*(E))$. To calculate the induced map on $K_0(C^*(E))$, observe that for each
$(a_i) \in \varinjlim(\ZZ E^0, A^t)$, we have $(a_i - A^ta_i)_{i=0}^\infty \in
\image\phi$, and so $(a_i) + \image\phi = (A^t a_i) + \image\phi$. So for $v \in E^0$,
\[
(j_m \circ \tilde{\iota}_m)_*(\omega(\delta_v)) + \image\phi
	= \sum^{m-1}_{j=0} ((A^t)^j\delta_v, (A^t)^{j+1}\delta_v, \dots) + \image\phi
	= m \cdot (\delta_v, A^t\delta_v, \dots)
	= m \cdot \omega(\delta_v),
\]
and we deduce that $(j_m \circ \tilde{\iota}_m)_*$ is multiplication by $m$ on
$K_1(C^*(E))$.
\end{proof}

It will be important for our main result that this inclusion restricts to an inclusion
map of the same form on gauge-invariant ideals, and induces an inclusion map of the same
form on the corresponding quotients.

\begin{lem}\label{lem:respects ideals}
Let $E$ be a row-finite directed graph with no sinks. If $I$ is a gauge-invariant ideal
of $C^*(E)$, then $(j_m \circ \tilde{\iota}_m)_*$ restricts to the multiplication-by-$m$
map $K_*(I) \to K_*(I)$ and induces the multiplication-by-$m$ map $K_*(C^*(E)/I) \to
K_*(C^*(E)/I)$.  Moreover, if $J$ is a gauge-invariant ideal of $C^*(E)$ with $J
\subseteq I$, then the induced map $K_* ( I / J ) \to K_* ( I / J )$ is multiplication by
$m$.
\end{lem}
\begin{proof}
Let $H = \{v \in E^0 : p_v \in I\}$. By \cite[Lemma~1.1]{BatesPaskEtAl:NYJM00}, the
multiplier projection $P_H := \sum_{v \in H} P_v$ is full in $I$, and there are canonical
isomorphisms $P_H C^*(E) P_H \cong C^*(EH)$ and $C^*(E)/I \cong C^*(E \setminus EH)$. Let
$H(m) := \{\mu \in E^{<m} : s(\mu) \in H\} \subseteq E(m)^0$. This $H(m)$ is a hereditary
set, and we have $H(m) E(m) = (HE)(m)$ as subgraphs of $E(m)$. So the proof of
\cite[Theorem~4.1(c)]{BatesPaskEtAl:NYJM00} shows that $P_{H(m)} := \sum_{\mu \in H(m)}
p_\mu \in \Mm( C^*(E(m)) )$ is a full projection in $I_{H(m)}$ and there is a canonical
isomorphism $P_{H(m)} C^*(E(m)) P_{H(m)} \cong C^*((HE)(m))$. The definition of the
homomorphism $j_m : C^*(E(m)) \to C^*(E) \otimes \Kk_{E^{<m}}$ shows that it restricts to
a homomorphism from $P_{H(m)} C^*(E(m)) P_{H(m)}$ to $P_H C^*(E) P_H \otimes
\Kk_{(HE)^{<m}}$. So, setting $Q_{H,m} := \sum_{\mu \in (HE)^{<m}} P_H \otimes
\theta_{\mu,\mu} \in \Mm(C^*(E) \otimes \Kk_{E^{<m}} )$, we see that the diagram
\[
\begin{tikzpicture}
    \node (HmEm) at (0,3) {$C^*(HE)$};
    \node (PEP) at (5,3) {$P_H C^*(E) P_H$};
    \node (E) at (10,3) {$C^*(E)$};
    \node (HmEmK) at (0,0) {$C^*(HE) \otimes \Kk_{(HE)^{<m}}$};
    \node (PEPK) at (5,0) {$Q_{H,m}  (C^*(E) \otimes \Kk_{E^{<m}}) Q_{H,m}$};
    \node (EK) at (10,0) {$C^*(E)$};
    \draw[-stealth] (HmEm)--(PEP) node[above, pos=0.5] {\small$\cong$};
    \draw[-stealth] (PEP)--(E) node[above, pos=0.5] {\small$\subseteq$};
    \draw[-stealth] (HmEmK)--(PEPK) node[above, pos=0.5] {\small$\cong$};
    \draw[-stealth] (PEPK)--(EK) node[above, pos=0.5] {\small$\subseteq$};
    \draw[-stealth] (HmEm)--(HmEmK) node[left, pos=0.5] {\small{$j^{HE}_m \circ \tilde{\iota}^{HE}_m$}};
    \draw[-stealth] (PEP)--(PEPK) node[left, pos=0.5] {\small{$j^{E}_m \circ \tilde{\iota}^{E}_m$}};
    \draw[-stealth] (E)--(EK) node[left, pos=0.5] {\small{$j^{E}_m \circ \tilde{\iota}^{E}_m$}};
\end{tikzpicture}
\]
commutes.

Lemma~\ref{lem:times m} implies that the vertical map on the left induces multiplication
by $m$ in $K$-theory, and so the map in the middle does too. This map is the restriction
of $j^{E}_m \circ \tilde{\iota}^{E}_m$ to $P_H C^*(E) P_H$. Since $P_H$ and $Q_{H,m}$ are
full in $I_H$ and $I_H \otimes \Kk_{E^{<m}}$, compression by these projections induces
isomorphisms in $K$-theory, so we deduce that the restriction of $j^{E}_m \circ
\tilde{\iota}^{E}_m$ to $I_H$ also induces the multiplication-by-$m$ map.

We have already showed that $j_m$ carries $I$ into the ideal $I_{H(m)} \lhd C^*(E(m))$
generated by $\{p_\mu : \mu \in H(m)\}$. Hence $j_m$ induces a homomorphism $\widetilde{j}_m
: C^*(E)/I \to C^*(E(m))/I_{H(m)}$. By \cite[Theorem~4.1]{BatesPaskEtAl:NYJM00} there is
a canonical isomorphism $C^*(E)/I \cong C^*(E \setminus EH)$. It is routine to check that
the saturation of $H(m)$ is the set $K = \{\mu \in E^{<m} : s(\mu) \in H\}$. So
\cite[Theorem~4.1]{BatesPaskEtAl:NYJM00} again implies that $C^*(E(m))/I_{H(m)}$ is
canonically isomorphic to $C^*(E(m) \setminus E(m)K)$. It is also routine to check that
$(E \setminus EH)(m) = E(m) \setminus E(m)K$ as subsets of $E(m)$. By comparing formulas
on generators, we see that the diagram
\[
\begin{tikzpicture}
    \node (E) at (0,3) {$C^*(E)$};
    \node (E/I) at (6,3) {$C^*(E)/I$};
    \node (E/H) at (12,3) {$C^*(E \setminus EH)$};
    \node (ExK) at (0,0) {$C^*(E) \otimes \Kk_{E^{<m}}$};
    \node (E/IxK) at (6,0) {$(C^*(E) \otimes \Kk_{E^{<m}})/(I \otimes \Kk_{E^{<m}})$};
    \node (E/HxK) at (12,0) {$C^*(E \setminus EH) \otimes \Kk_{E^{<m}}$};
    \draw[-stealth] (E)--(E/I) node[above, pos=0.5] {\small$q_I$};
    \draw[-stealth] (E/I)--(E/H) node[above, pos=0.5] {\small$\cong$};
    \draw[-stealth] (ExK)--(E/IxK) node[above, pos=0.5] {\small$q_{I \otimes \Kk_{E^{<m}}}$};
    \draw[-stealth] (E/IxK)--(E/HxK) node[above, pos=0.5] {\small$\cong$};
    \draw[-stealth] (E)--(ExK) node[left, pos=0.5] {\small{$j^{E}_m \circ \tilde{\iota}^{E}_m$}};
    \draw[-stealth] (E/I)--(E/IxK) node[left, pos=0.5] {\small{$(j^{E}_m \circ \tilde{\iota}^{E}_m)^{\sim}$}};
    \draw[-stealth] (E/H)--(E/HxK) node[left, pos=0.5] {\small{$j^{E\setminus EH}_m \circ \tilde{\iota}^{E\setminus EH}_m$}};
\end{tikzpicture}
\]
commutes. Since $C^*(E \setminus EH) \otimes \Kk_{(E\setminus EH)^{<m}}$ embeds as a full
corner in $C^*(E \setminus EH) \otimes \Kk_{E^{<m}}$, the corresponding inclusion map in
$K$-theory is an isomorphism. So Lemma~\ref{lem:times m} implies that the vertical
inclusion map $j^{E\setminus EH}_m \circ \tilde{\iota}^{E\setminus EH}_m$ on the right of
the diagram induces multiplication by $m$ in $K$-theory. Thus the middle vertical map
does too. Since this is precisely the map on the quotient induced by $j^{E}_m \circ
\tilde{\iota}^{E}_m$, we deduce that $(j_m \circ \tilde{\iota}_m)_*$ restricts to $m \cdot
\id$ on $K_*(I)$ and induces $m \cdot \id$ on $K_*(C^*(E)/I)$.

For the final assertion, apply the preceding assertion to the gauge-invariant ideal $P_H
J P_H$ of $P_H C^*(E) P_H \cong C^*(E \setminus EH)$, using that compression by the full
projection $P_H$ induces an isomorphism of the $K$-theory of $I$.
\end{proof}

\section{A technique of Enders}\label{sec:enders}

In this section we apply general results of Meyer-Nest and of Kirchberg to generalize a
technique of Enders to see that, under suitable hypotheses, an endomorphism $\kappa_m$ of
a strongly purely infinite nuclear $C^*$-algebra $A$ whose image $\FK(\kappa)$ in
filtered $K$-theory is the times-$m$ map can be post-composed with an automorphism of $A$
so that the resulting map approximates the identity. We apply this result to prove the
following theorem (the proof appears at the end of the section), which is our main
result.

\begin{thm}\label{thm:pi graph alg}
Let $E$ be a directed graph. Suppose that $C^*(E)$ is purely infinite and has finitely
many ideals. Then the nuclear dimension of $C^*(E)$ is 1.
\end{thm}

We start by recalling some terminology and background from \cite{MeyerNest:MJM09}.

We write $\Prim(A)$ for the primitive-ideal space of a $C^*$-algebra $A$, and equip
$\Prim(A)$ with the hull-kernel topology. Recall from
\cite[Definition~2.3]{MeyerNest:MJM09} that if $X$ is a topological space, then a
$C^*$-algebra over $X$ is a pair $(A, \psi)$ consisting of a $C^*$-algebra $A$ and a
continuous map $\psi : \Prim(A) \to X$. As in \cite[Definition~5.1]{MeyerNest:MJM09}, we
say that $(A,\psi)$ is a tight $C^*$-algebra over $X$ if $\psi$ is a homeomorphism.

Let $(A,\psi)$ be a $C^*$-algebra over $X$. We write $\OO(X)$ for the lattice of open
subsets of $X$. Given $U \in \OO(X)$, we write $A(U)$ for the corresponding ideal
$\bigcap \{I \in \Prim(A) : \psi(I) \not\in U\}$. We then have $A(U \cup V) = A(U) +
A(V)$ and $A(U \cap V) = A(U) \cap A(V)$ for $U,V \in \OO(X)$ (in fact, a stronger
statement is true \cite[Lemma~2.8]{MeyerNest:MJM09}).

As in \cite[Definition~2.10]{MeyerNest:MJM09}, if $(A, \psi)$ and $(B, \rho)$ are
$C^*$-algebras over $X$, we say that a homomorphism $\phi : A \to B$ is
\emph{$X$-equivariant} if $\phi(A(U)) \subseteq B(U)$ for all $U \in \OO(X)$. If, in
addition, $A$ is a tight $C^*$-algebra over $X$, we say that $\phi$ is \emph{full} if,
whenever $a$ generates $A(U)$ as an ideal, $\phi(a)$ generates $B(U)$ as an ideal.

Given homomorphisms $\phi, \psi : A \to B$ between $C^*$-algebras, we say that $\phi$ and
$\psi$ are asymptotically unitarily equivalent, and write $\phi \aua \psi$, if there is a
continuous family of unitaries $U_t$  ($t \in [1, \infty )$) in $\mathcal{M}(B)$ such that $\Ad U_t \circ \phi(a) \to
\psi(a)$ as $t \to \infty$ for all $a \in A$.  We say that $\phi$ and $\psi$ are approximately unitarily equivalent, and write $\phi \approx_{u} \psi$, if there exists a sequence $U_{n}$ of unitaries in $\mathcal{M}(B)$ such that $U_{n} \phi ( a ) U_{n}^{*} \to \psi ( a )$ for all $a \in A$.

\begin{lem}\label{lem:jXA properties}
Let $X$ be a finite topological space, and let $(A,\psi)$ be a tight $C^*$-algebra over
$X$. Fix a unital homomorphism $j_1 : \Oo_\infty \to \Oo_2$ and a nonzero homomorphism
$j_2 : \Oo_2 \to \Oo_\infty$, and let $j = (j_2 \circ j_1) \otimes \id_\Kk : \Oo_\infty
\otimes \Kk \to \Oo_\infty \otimes \Kk$. Suppose that $A$ is separable and nuclear, and
define $j_X^A := \id_A \otimes j : A \otimes \Oo_\infty \otimes \Kk \to A \otimes
\Oo_\infty \otimes \Kk$. Then
\begin{enumerate}
\item\label{it:id+j=id} $\id_A \oplus j_X^A \aua \id_{A \otimes \Oo_\infty \otimes
    \Kk}$;
\item\label{it:j+j=j} $j_X^A \oplus j_X^A \aua j_X^A$; and
\item\label{it:full+equivariant} $j_X^A$ is a full $X$-equivariant homomorphism.
\end{enumerate}
If $(B, \rho)$ is a second tight $C^*$-algebra over $X$ and if $\phi : A \otimes
\Oo_\infty \otimes \Kk \to B \otimes \Oo_\infty \otimes \Kk$ is a homomorphism, then
$(j_X^B \circ \phi) \oplus (j_X^B \circ \phi) \aua j_X^B \circ \phi$.
\end{lem}
\begin{proof}
Since $1 \otimes \Mm(\Kk) \cong 1 \otimes \Bb(\ell^2)$ is contained in $\Mm(\Oo_\infty
\otimes \Kk)$, there exist isometries $s_1, s_2 \in \Mm(\Oo_\infty \otimes \Kk)$ such
that $s_1 s^*_1 + s_2 s^*_2 = 1_{\Mm(\Oo_\infty \otimes \Kk)}$. Hence the isometries $t_i
:= 1_{\Mm(A)} \otimes s_i \in \Mm(A \otimes \Oo_\infty \otimes \Kk)$ satisfy $t_1 t^*_1 +
t_2 t^*_2 = 1$ as well.

We have $KK(j_1) \in KK(\Oo_\infty, \Oo_2) = \{0\}$, and hence $KK(j) = 0$. So
\cite[Theorem~4.13]{Phillips:DM00} implies that
\begin{itemize}
    \item[(j1)] $\id_{\Oo_\infty \otimes \Kk} \oplus j \aua \id_{\Oo_\infty \otimes
        \Kk}$,
    \item[(j2)] $j(a \otimes x) = j_2(1) a j_2(1) \otimes x$, and
    \item[(j3)] $j \oplus j \aua j$.
\end{itemize}

We have
\[
\id_A \oplus j_X^A
    = \Ad(t_1)\circ \id_{A \otimes \Oo_\infty \otimes \Kk} + \Ad(t_2)\circ j_X^A\\
    = \id_A \otimes \big(\Ad(s_1) \circ \id_{\Oo_\infty \otimes \Kk}
        + \Ad(s_2) \circ j\big).
\]
Property~(j1) of $j$ therefore gives $\id_A \oplus j_X^A \aua \id_A \otimes
\id_{\Oo_\infty \otimes \Kk} = \id_{A \otimes \Oo_\infty \otimes \Kk}$. Likewise, using
property~(j2) at the last step, we obtain
\begin{align*}
j_X^A \oplus j_X^A
    &= \Ad(t_1)\circ j_X^A + \Ad(t_2)\circ j_X^A \\
    &= \id_A \otimes \big(\Ad(s_1) \circ j
        + \Ad(s_2) \circ j\big)
    \aua \id_A \otimes j = j_X^A.
\end{align*}
For $U \in \OO(X)$, we have
\begin{align*}
j_X^A((A \otimes \Oo_\infty \otimes \Kk)(U))
    &= A(U) \otimes j(\Oo_\infty \otimes \Kk) \\
    &= A(U) \otimes j_2(1)\Oo_\infty j_2(1) \otimes \Kk
    \subseteq A(U) \otimes \Oo_\infty \otimes \Kk,
\end{align*}
so that $j_X^A$ is $X$-equivariant. To see that it is full, fix a full element $y \in (A
\otimes \Oo_\infty \otimes \Kk)(U)$. Fix a full element $a \in A(U)$ and let $e_{ij}$ be
the canonical matrix units in $\Kk$. Then $a \otimes 1 \otimes e_{11}$ belongs to the
ideal generated by $y$, and so $a \otimes j_2(1) \otimes e_{11}$ belongs to the ideal
generated by $j_X^A(y)$. Since $j_2(1)$ is nonzero and $\Oo_\infty$ and $\Kk$ are simple,
the ideal generated by $a \otimes j_2(1) \otimes e_{11}$ is $\overline{AaA} \otimes
\Oo_\infty \otimes \Kk = (A \otimes \Oo_\infty \otimes \Kk)(U)$.

Now suppose that $(B, \rho)$ is a second tight $C^*$-algebra over $X$ and that $\phi : A
\otimes \Oo_\infty \otimes \Kk \to B \otimes \Oo_\infty \otimes \Kk$ is a homomorphism.
The above argument shows that $j_X^B \oplus j_X^B \aua j_X^B$, so there exists a continuous family of unitaries
$U_t$ such that $U_t (j_X^B \oplus j_X^B)(b) U^*_t \to j_X^B(b)$ for all $b$. Hence
\[
U_t\big((j_X^B \circ \phi) \oplus (j_X^B \circ \phi)\big)(a)U^*_t
    = U_t\big(j_X^B \oplus j_X^B\big)(\phi(a))U^*_t
    \to j_X^B(\phi(a)).\qedhere
\]
\end{proof}

\begin{lem}\label{lem:magic aua}
Suppose that $(A,\psi)$ and $(B,\rho)$ are separable, nuclear, tight $C^*$-algebras over
a finite topological space $X$. For $i = 1,2$, suppose that $\phi_i : A \otimes
\Oo_\infty \otimes \Kk \to B \otimes \Oo_\infty \otimes \Kk$ is a full $X$-equivariant
homomorphism such that $\phi_i \oplus \phi_i \aua \phi_i$. Then $\phi_1 \aua \phi_2$.
\end{lem}
\begin{proof}
This follows from \cite[Hauptsatz~2.15]{Kirchberg}.
\end{proof}

We can now deduce that, amongst full equivariant homomorphisms between strongly purely
infinite $C^*$-algebras, asymptotic unitary equivalence is characterized by equivariant
$KK$-equivalence.

\begin{thm}\label{thm:KK<->aua}
Let $(A,\psi)$ and $(B,\rho)$ be separable, nuclear, tight $C^*$-algebras over a finite
topological space $X$. For $i = 1,2$ let $\phi_i : A \otimes \Oo_\infty \otimes \Kk \to B
\otimes \Oo_\infty \otimes \Kk$ be a full $X$-equivariant homomorphism. Then $KK(X;
\phi_1) = KK(X; \phi_2)$ if and only if $\phi_1 \aua \phi_2$.

Consequently, if $KK(X; \phi_1) = KK(X; \phi_2)$, then $\phi_{1} \approx_{u} \phi_{2}$.
\end{thm}
\begin{proof}
If $\phi_1 \aua \phi_2$, then they are $KK(X;\cdot)$-equivalent by definition of the
$KK(X; \cdot)$ functor.

Suppose that $KK(X; \phi_1) = KK(X; \phi_2)$. By \cite[Hauptsatz~4.2]{Kirchberg}, there
are full $X$-equivariant homomorphisms $h_1, h_2 : A \otimes \Oo_\infty \otimes \Kk \to B
\otimes \Oo_\infty \otimes \Kk$ such that $h_i \oplus h_i \aua h_i$ for $i = 1,2$ and such that
$\phi_1 \oplus h_1 \aua \phi_2 \oplus h_2$.

The final statement of Lemma~\ref{lem:jXA properties} implies that each $j_X^B \circ
\phi_i$ satisfies $j_X^B \circ \phi_i \oplus j_X^B \circ \phi_i \aua j_X^B \circ \phi_i$,
and so we may apply Lemma~\ref{lem:magic aua} twice to see that $j_X^B \circ \phi_i \aua
h_i$ for $i = 1,2$. Lemma~\ref{lem:jXA properties}(\ref{it:id+j=id}) implies that $\id_B
\oplus j_X^B \aua \id_B$, and so for each $i$ we have
\[
\phi_i \aua (\id_B \oplus j_X^B)\circ \phi_i
    = \phi_i \oplus j_X^B \circ \phi_i
    \aua \phi_i \oplus h_i.
\]
Since $\phi_1 \oplus h_1 \aua \phi_2 \oplus h_2$, the result follows.
\end{proof}

We now have the material we need to extend Enders' technique to purely infinite
$C^*$-algebras with finite primitive-ideal spaces, leading to the proof of our main
result about purely infinite graph $C^*$-algebras. Recall that given a finite topological
space $X$, the filtered-$K$-theory functor $FK$ sends each $C^{*}$-algebra $A$ over $X$
to the $K$-groups of the subquotients of $A$ corresponding to locally closed subsets of
$X$, together with all the natural transformations between these groups. Morphisms
between $FK(A)$ and $FK(B)$ respect the natural transformations induced by $KK$ elements
between the subquotients of $A$ and $B$ corresponding to any given locally closed subset
of $X$. For details, see~\cite{MeyerNest:CJM12}.

\begin{prp}\label{prp:Enders technique}
Suppose that $(A,\phi)$ is a tight $C^*$-algebra over a finite topological space $X$.
Suppose that $A$ is stable, nuclear, separable and purely infinite. Suppose that there is
a sequence $\kappa_m$ of full $X$-equivariant homomorphisms $\kappa_m : A \to A$ such
that
\begin{enumerate}
\item $FK(\kappa_m) = m \cdot FK(\id_A)$ and
\item the sequence $(\kappa_{m})$ has an asymptotic order-$n$ factorization
	through elements of $\mathcal{C}$.
\end{enumerate}
Then $\id_{A}$ has an asymptotic order-$n$ factorization through elements of
$\mathcal{C}$.
\end{prp}
\begin{proof}
Using \cite[Theorem~4.3]{TomsWinter:TAMS07} and
\cite[Theorem~3.15]{KirchbergPhillips:Crelles00}, and a simple induction argument, one
can show that every separable, nuclear, purely infinite $C^*$-algebra with finitely many
ideals is $\mathcal{O}_\infty$-absorbing.  Thus, $A \cong A \otimes
\mathcal{O}_{\infty}$.  Let $(F_m, \psi_m, \varphi_m)$ be an asymptotic order-$n$
factorization through elements of $\mathcal{C}$.  The element $x = -KK(X;\id_A)$ is a
$KK$-equivalence, so \cite[Folgerung~4.3]{Kirchberg} implies that there is an
automorphism $\lambda$ of $A$ such that $KK(X; \lambda) = -KK(X; \id_A)$. We then have
$FK(\lambda) = -FK(\id_A)$ because the filtered $K$ functor is compatible with $KK$.
Since $A$ is stable, $\Mm(A)$ contains isometries $s_1$ and $s_2$ such that $s_1s^*_1 +
s_{2} s_{2}^{*} = 1$. Define $\beta_m : A \to A$ by $\beta_m(a) = s_1 \kappa_{m+1}(a)
s^*_1 + s_2 \lambda (\kappa_m(a)) s^*_2$. Since each $\kappa_l$ is a full $X$-equivariant
homomorphism, the same is true of each $\beta_m$, and we have
\[
FK(\beta_m)
    = FK(\kappa_{m+1}) + FK(\lambda \circ \kappa_m)
    = (m+1)FK(\id_A) - mFK(\id_A)
    = FK(\id_A).
\]
Hence \cite[Proposition~4.15]{MeyerNest:MJM09} implies that $KK(X;\beta_m)$ is a
$KK(X;\cdot)$-equivalence. Now \cite[Folgerung~4.3]{Kirchberg} implies that there is an
automorphism $\gamma_m$ of $A$ such that $KK(X;\gamma_m \circ \beta_m) = KK(X; \id_A)$.
Theorem~\ref{thm:KK<->aua} then implies that $\gamma_m \circ \beta_m \approx_{u} \id_A$; fix a
sequence $u_n \in \mathcal{M} ( A )$ of unitaries implementing the approximate unitary equivalence, so
$\|u_n \gamma_n \circ \beta_n(a) u^*_n - a\| \to 0$ for each $a \in A$.

Write $\Psi_m := \psi_m \oplus \psi_{m+1} : A \to F_m \oplus F_{m+1}$; this is a
completely positive contraction because each $\psi_i$ is. Define $\Phi_m : F_m \oplus
F_{m+1} \to A$ by
\[
\Phi_m(a,b) = u_m\gamma_m\big(s_1\varphi_{m+1}(b)s^*_1
    + s_2\lambda(\varphi_m(a))s^*_2\big)u^*_m.
\]
Each of $\phi_m$ and $\varphi_{m+1}$ is completely positive by definition, and then
$\lambda \circ \varphi_m$ is completely positive too because $\lambda$ is a homomorphism.
Conjugation by an isometry is likewise completely positive, so $(a,b) \mapsto
s_1\varphi_{m+1}(b)s^*_1 + s_2\lambda(\varphi_m(a))s^*_2$ is also completely positive.
Since $\Ad(u_m) \circ \gamma_{m}$ is a homomorphism, it follows that $\Phi_m$ is
completely positive. Since $\lambda$, $\Ad(u_m)$ and $\gamma_{m}$ are homomorphisms, and
therefore norm-decreasing, and since $s_1 s^*_1 \perp s_2 s^*_2$, we have
$\|\Phi_m(a,b)\| = \max\{\|\varphi_{m+1}(b)\|, \|\varphi_m(a)\|\}$; since $\varphi_{m+1}
\oplus \varphi_m$ is contracting on each $F_m^{(k)} \oplus F_{m+1}^{(k)}$, the same is
true of $\Phi_m$. Fix $k \in \{0,1, \dots, n\}$ and suppose that $a = (a_m, a_{m+1})$ and
$b = (b_m, b_{m+1})$ are orthogonal elements of $F_m^{(k)} \oplus F_{m+1}^{(k)} \subseteq
F_m \oplus F_{m+1}$; so $a_i b_i = b_ia_i = 0$ for $i = m,m+1$. Since $s_i^*s_j =
\delta_{i,j} 1$, we have
\[
\gamma_m^{-1}(u_m^*\Phi_m(a)\Phi_m(b)u_m)
    = s_1 \varphi_{m+1}(a_{m+1})\varphi_{m+1}(b_{m+1}) s^*_1
        + s_2 \lambda(\varphi_m(a_m)\varphi_m(b_m)) s^*_2 = 0.
\]
Since $\gamma_m^{-1}$ and $\Ad(u_m^*)$ are automorphisms, we deduce that
$\Phi_m(a)\Phi_m(b) = 0$, and symmetrically for $\Phi_m(b)\Phi_m(a)$. So $\Phi_m$
restricts to an order-0 contraction on $F_m^{(k)} \oplus F_{m+1}^{(k)}$.

Fix $a \in A$. It now suffices to show that $\|\Phi_m \circ \Psi_m(a) - a\| \to 0$. By
definition of $\Phi_m$ and $\Psi_m$, we have
\[
\Phi_m(\Psi_m(a))
    = u_m\gamma_m\big(s_1\varphi_{m+1}(\psi_{m+1}(a))s^*_1
        + (s_2\lambda(\varphi_m(\psi_m(a)))s^*_2)\big)u^*_m.
\]
Hence
\begin{align*}
\|\Phi_m&(\Psi_m(a)) - a\|\\
     &\le \big\|\Phi_m(\Psi_m(a))
        - u_m\gamma_m\big(s_1 \kappa_{m+1}(a) s^*_1 + s_2 \lambda(\kappa_m(a))\big)u_m^*\big\|\\
     &\hskip2cm{}+ \big\|u_m\gamma_m\big(s_1 \kappa_{m+1}(a) s^*_1 + s_2 \lambda(\kappa_m(a))\big)u_m^*
            - a\big\|\\
     &= \Big\|u_m\gamma_m\Big(s_1 \big(\varphi_{m+1}\circ\psi_{m+1}(a) - \kappa_{m+1}(a)\big) s^*_1
                + s_2\lambda\big(\varphi_m\circ\psi_m(a) - \kappa_m(a)\big)s^*_2\Big)u_m^*\Big\|\\
        &\hskip2cm{}+ \|u_m\gamma_m\circ \beta_m(a)u_m^*
            - a\|\\
     &= \max\big\{\|\varphi_{m+1}\circ\psi_{m+1}(a) - \kappa_m(a)\|,
                    \|\varphi_m\circ\psi_m(a) - \kappa_m(a)\|\big\} \\
        &\hskip2cm{}+ \|u_m\gamma_m\circ \beta_m(a)u_m^* - a\| \\
     &\to 0,
\end{align*}
completing the proof.
\end{proof}

\begin{proof}[Proof of Theorem~\ref{thm:pi graph alg}]
Let $E'$ be a Drinen-Tomforde desingularization of $E$ as in
\cite{DrinenTomforde:RMJM05}, and let $G$ be a graph as in \cite{Tomforde:PAMS04} such
that $C^*(G) \cong C^*(E') \otimes \Kk$. The nuclear dimension of $C^*(G)$ is equal to
that of $C^*(E)$ by \cite[Corollary~2.8(i)]{WinterZacharias:AM10}. Since $C^*(E)$ is not
an AF-algebra, the nuclear dimension of $C^{*} (E)$ (and hence the nuclear dimension of
$C^*(G)$) is at least 1 (see \cite[Remarks~2.2(iii)]{WinterZacharias:AM10}). We must show
that it is at most 1.

Since $G$ is row-finite, Corollary~\ref{cor:order-zero-approximation} implies that
$\tilde{\iota}_m$ has an asymptotic order-$1$ factorization through finite dimensional
$C^{*}$-algebras. Using again that $G$ is row-finite and has no sinks, we can apply
Proposition~\ref{prp:E(m)->ExK} to obtain a map $j_m : C^*(G(m)) \to C^*(G) \otimes
\Kk_{G^{<m}}$ satisfying~\eqref{eq:jsubmdef}.  Since $C^*(G)$ is stable, there is an
isomorphism $\rho : C^*(G) \otimes \Kk_{G^{<m}} \to C^*(G)$ such that $FK(\rho) = \id$.
By the construction of $\tilde{\iota}_{m}$ and $j_{m}$, the map $j_{m} \circ
\tilde{\iota}_{m}$ is an $X$-equivariant homomorphism and $j_{m} \circ
\tilde{\iota}_{m}(a)$ is full in $C^{*}(E)(U)$ whenever $a$ is full in $C^{*} (E)(U)$.
Thus, $j_{m} \circ \widetilde\iota_{m}$ is a full $X$-equivariant homomorphism.
Lemma~\ref{lem:respects ideals} implies that the maps $\kappa_m := \rho \circ j_m \circ
\widetilde\iota_m : C^*(G) \to C^*(G)$ satisfy $FK(\kappa_m) = m\cdot FK(\id_{C^*(G)})$.
Since $j_{m}$ and $\rho$ are $*$-homomorphisms, $\kappa_m$ has an asymptotic order-$1$
factorization through finite dimensional $C^{*}$-algebras.  So
Proposition~\ref{prp:Enders technique} implies that $\id_{C^{*} (G)}$ has an asymptotic
order-$1$ factorization through finite dimensional $C^{*}$-algebras, and thus the nuclear
dimension of $C^*(G)$ is at most 1.
\end{proof}

\begin{cor}
Let $X$ be an accordion space as defined in \cite{BentmannKohler:xx11} and let $A$ be a
separable, nuclear, purely infinite, tight $C^{*}$-algebra over $X$. Suppose that
$K_{1} ( A(x) )$ is free abelian and that $A(x)$ is in the Rosenberg-Schochet bootstrap
category $\mathcal{N}$ for each $x \in X$.  Then the nuclear dimension of $A$ is $1$.
\end{cor}
\begin{proof}
By \cite[Corollary~7.16]{ArklintBentmannKatsura:xx13inv} (see also
\cite{ArklintBentmannKatsura:xx13range}), there exists a row-finite directed graph $E$
such that $C^*(E) \otimes \mathcal{K} \cong A \otimes \mathcal{K}$. Now
\cite[Corollary~2.8(i)]{WinterZacharias:AM10} and Theorem~\ref{thm:pi graph alg} imply
that the nuclear dimension of $A$ is $1$.
\end{proof}

\section{Purely infinite graph $C^*$-algebras with infinitely many ideals}\label{sec:consequences}
In \cite{JeongPark:JFA02}, Jeong and Park show how to describe the $C^*$-algebra of a
row-finite graph with no sources as a direct limit of $C^*$-algebras of graphs with
finitely many ideals. Here we adapt their technique to see that if $E$ satisfies
Condition~(K) and every vertex of $E$ connects to a cycle, then $C^*(E)$ has nuclear
dimension at most~2.

By a \emph{first-return path} at a vertex $v$ in a directed graph $E$, we mean a path
$e_1 \dots e_n$ such that $r(e_1) = s(e_n) = v$ and $r(e_i) \not= v$ for $i \ge 2$.
Recall that a graph $E$ satisfies Condition~(K) if, whenever there is a first-return path
in $E$ at $v$, there are at least two distinct first-return paths at $v$ (see
\cite[Section~6]{KPRR}).

\begin{thm}\label{thm:infgraphs} Let $E$ be a graph that satisfies Condition (K), and
suppose that each vertex of $E$ connects to a cycle in $E$. Then $1 \le \nucdim(C^*(E))
\le \nucdim(\Tt C^*(E)) \le 2$.
\end{thm}
\begin{proof}
Since $E$ contains cycles, $C^*(E)$ is not AF and therefore has nuclear dimension at
least 1 by \cite[Remark~2.2(iii)]{WinterZacharias:AM10}. Since $C^*(E)$ is a quotient of
$\Tt C^*(E)$, Proposition~2.9 of \cite{WinterZacharias:AM10} implies that
$\nucdim(C^*(E)) \le \nucdim(\Tt C^*(E))$. It remains to show that $\nucdim(\Tt C^*(E))
\le 2$.

Fix a finite set $V \subseteq E^0$ and a finite set $F \subseteq E^1$. We claim that
there exists a finite subgraph $G_{V,F}$ of $E$ such that $V \subseteq G_{V,F}^0$, $F
\subseteq G_{V,F}^1$, every vertex in $G_{V,F}$ connects to a cycle in $G_{V,F}$, and
$G_{V,F}$ satisfies Condition~(K). Since every vertex in $E$ connects to a cycle, there
exists, for each $v \in V \cup r(F)$, a path $\lambda^v = \lambda^v_1 \dots
\lambda^v_{|\lambda^v|}$ in $vE^*$ such that $s(\lambda^v_n) = r(\lambda^v)$ for some $n
\le |\lambda^v|$. Let $G$ be the subgraph of $E$ given by $G_0^1 := F \cup \{\lambda^v_i
: v \in V \cup r(F), i \le |\lambda^v|\}$, and let $G_0^0 := r(G_0^1) \cup s(G_0^1)$. By
construction, every vertex in $G_0$ connects to a cycle in $G_0$. For each $w \in G_0^0$
that lies on exactly one cycle, say $\mu^w$, in $G_0$, condition~(K) in $E$ implies that
there exists $n \le |\mu^w|$ and a cycle $\nu^w$ in $E$ based at $s(\mu^w_n)$ such that
$\nu^w_1 \not= \mu^w_n$. We let $G_{V,F}^1 = G_0^1 \cup \{s(\nu^w_i) : w\text{ lies on
exactly one cycle in $G_0$}, 1 \le i \le |\nu^w|\}$, and $G_{V,F}^0 = r(G_{V,F}^1) \cup
s(G_{V,F}^1)$. Then $G_{V,F}$ has all the desired properties, establishing the claim.

Proposition~5.3 of \cite{BatesPaskEtAl:NYJM00} implies that $C^*(G_{V,F})$ is purely
infinite, and \cite[Theorem~4.4]{BatesPaskEtAl:NYJM00} implies that it has finitely many
ideals. So Theorem~\ref{thm:pi graph alg} implies that $C^*(G_{V,F})$ has nuclear
dimension 1.

The kernel of the quotient map $q : \Tt C^*(G_{V,F}) \to C^*(G_{V,F})$ is AF. (This is
well-known, but to verify it, recall that the path-space representation $\pi_{T,Q}$ of
$\Tt C^*(E)$ on $\ell^2(E^*)$ is faithful by \cite[Theorem~4.1]{FowlerRaeburn:IUMJ99},
and observe that each $\pi_{T,Q}(q_v - \sum_{e \in vE^1} t_e t^*_e)$ is the rank-one
projection onto the basis element $\delta_v$, so that $\pi_{T,Q}(\ker q)$ is contained in
$\Kk_{E^*}$.) Hence Remark~2.2 of \cite{WinterZacharias:AM10} (see also Example~4.1 of
\cite{KirchbergWinter}) implies that $\nucdim(\ker q) = 0$. So $\Tt C^*(G_{V,F})$ is an
extension of an algebra of nuclear dimension~1 by an algebra of nuclear dimension~0, and
\cite[Proposition~2.9]{WinterZacharias:AM10} then implies that $\Tt C^*(G_{V,F})$ has
nuclear dimension at most $2$.

The generators $\{q_v : v \in G_{V,F}^0\} \cup \{t_f : f \in G_{V,F}^1\}$ constitute a
Toeplitz-Cuntz-Krieger $G_{V,F}$-family in $\Tt C^*(E)$, and so induce a homomorphism
$\iota_{G_{V,F}} : \Tt C^*(G_{V,F}) \to \Tt C^*(E)$. Since the preceding paragraph
implies that $\Tt C^*(G_{V,F})$ has nuclear dimension at most 2, and since
$\iota_{G_{V,F}}(\Tt C^*(G_{V,F}))$ is isomorphic to a quotient\footnote{In fact, [12,
Theorem 4.1] can be used to show that $\iota_{G, V}$ is injective, but this is not
necessary for our argument.} of $\Tt C^*(G_{V,F})$, Proposition~2.9 of
\cite{WinterZacharias:AM10} implies that $\iota_{G,V}(\Tt C^*(G_{V,F}))$ also has nuclear
dimension at most 2. Hence $\Tt C^*(E) = \overline{\bigcup_{V,F} \iota_{G_{V,F}}(\Tt
C^*(G_{V,F}))}$ is a direct limit of $C^*$-algebras of nuclear dimension at most 2, and
then \cite[Proposition~2.3(iii)]{WinterZacharias:AM10} implies that $\Tt C^*(E)$ has
nuclear dimension at most~2.
\end{proof}

\begin{cor}\label{cor:old pi nucdim}
Suppose that $E$ is a directed graph such that $C^*(E)$ is purely infinite. Then
$\nucdim(C^*(E)) \le 2$.
\end{cor}
\begin{proof}
By (b)~implies~(e) of \cite[Theorem~2.3]{HongSzymanski:BLMS03}, the graph $E$ satisfies
Condition~(K) and every vertex in $E$ connects to a cycle. So the result follows from
Theorem~\ref{thm:infgraphs}.
\end{proof}

\section{Quasidiagonal extensions and nuclear dimension}\label{sec:qdext}

In this section we show that the nuclear dimension of a quasidiagonal extension $0 \to I
\to A \to A/I \to 0$ is equal to the maximum of the nuclear dimension of $I$ and the
nuclear dimension of $A/I$. By showing that ideals in graph $C^*$-algebras for which the
quotient is AF are quasidiagonal (we are indebted to James Gabe for this argument), we
deduce the following extension of Theorem~\ref{thm:pi graph alg}.

\begin{thm}\label{thm:AF-ext}
Suppose that $E$ is a directed graph, that $I$ is a purely infinite gauge-invariant ideal
of $C^*(E)$, and that $C^*(E)/I$ is AF. Then $\nucdim(C^*(E)) \le 2$. If $I$ has
finitely many ideals, then $\nucdim(C^*(E)) = 1$.
\end{thm}

We prove Theorem~\ref{thm:AF-ext} at the end of the section. Prior to the proof we need to establish
some preliminary results on quasidiagonal extensions.

\begin{dfn}\label{dfn:qde}
Let $A$ be a separable $C^*$-algebra and $I$ an ideal of $A$. We say that $0 \to I \to A
\to A/I \to 0$ is a \emph{quasidiagonal extension} if there is an approximate identity
$(p_n)^\infty_{n=1}$ of projections in $I$ such that $\|p_n a - ap_n\| \to 0$ for all $a
\in A$. We refer to the sequence $(p_n)$ as \emph{quasicentral} in $A$.
\end{dfn}

\begin{prp}\label{p:nucleardim-quasi}
Suppose that $A$ is a separable nuclear $C^*$-algebra and $0 \to I \to A \to A/I \to 0$
is a quasidiagonal extension. Then $\nucdim(A) = \max\{\nucdim(I), \nucdim(A/I)\}$.
\end{prp}
\begin{proof}
We modify the argument of \cite[Proposition~6.1]{KirchbergWinter} using the first part of
the argument of \cite[Proposition~2.9]{WinterZacharias:AM10}.

Suppose that both $\nucdim(A/I)$ and $\nucdim(I)$ are at most $d$. We must show that
$\nucdim(A) \le d$. Fix normalized positive elements $e_1, \dots, e_L$ of $A$, and fix
$\varepsilon > 0$.

Choose a finite-dimensional $F = \bigoplus^d_{i=0} F^{(i)}$, a completely positive
contraction $\eta : A/I \to F$ and a completely positive $\rho : F \to A/I$ which is an
order-0 contraction on each $F^{(i)}$ so that $\|\rho(\eta(q_I(e_l))) - q_I(e_l)\| <
\varepsilon/5$ for each $l$. Let $P = \{p_\alpha\}$ be an approximate identity of
projections in $I$ which is quasicentral in $A$. We follow the proof of
\cite[Proposition~2.9]{WinterZacharias:AM10} as far as the middle of the top of page~472
first to lift $\rho$ to a completely positive $\overline{\rho} : F \to A$ which is an
order-0 contraction on each $F^{(i)}$, and then to find $p \in P$ such that, with
$\delta$ as in \cite[Proposition~2.6]{KirchbergWinter},
\begin{enumerate}
\item\label{it:almost commutes} $\|[(1-p), \overline{\rho}(x)]\| < \delta\|x\|$ for
    all $x \in F$,
\item each $\|(p e_l p + (1-p)e_l(1-p)) - e_l\| < \varepsilon/5$, and
\item each $\big\|(1-p)\big(\overline{\rho}(\eta(e_l)) - e_l\big)(1-p)\big\| <
    2\varepsilon/5$.
\end{enumerate}
Now using Proposition~1.4 of \cite{WinterZacharias:AM10} and property~(\ref{it:almost
commutes}) of $\overline{\rho}$ above, we obtain completely positive order-0 contractions
$\hat\rho^{(i)} : F^{(i)} \to (1-p)A(1-p)$ such that, putting $\hat\rho =
\sum_i\hat\rho^{(i)}$, we have
\[
    \|\hat\rho(x) - (1-p)\overline{\rho}(x)(1-p)\| < \varepsilon\|x\|/5
        \quad\text{ for all $x \in F$.}
\]

Now choose a finite-dimensional $G = \bigoplus^d_{i=0} G^{(i)}$, a completely positive
contraction $\psi : I \to G$ and a completely positive $\varphi : G \to I$ which is an
order-0 contraction on each $G^{(i)}$ such that $\|\varphi(\psi(pe_l p)) - p e_l p\| <
\varepsilon/5$ for each $l$. Define $H = \bigoplus^d_{i=1} (F^{(i)} \oplus G^{(i)})$, and
define $\Phi : A \to F \oplus G \cong H$ by $a \mapsto \big(\eta(q_I(a)),
\psi(pap)\big)$. For each $i$, define $\Phi^{(i)} : F^{(i)} \oplus G^{(i)} \to A$ by
$\Phi^{(i)}(x,y) = \hat\rho(x) + \varphi(y)$. Since each $\hat\rho(x) \in (1-p)A(1-p)$
and each $\varphi(y) \in pAp$, the $\Phi^{(i)}$ are order-zero contractions. The final
calculation of the proof of \cite[Proposition~2.9]{WinterZacharias:AM10} applies verbatim
to show that $\|\Phi(\Psi(e_l)) - e_l\| < \varepsilon$ for all $l$.
\end{proof}

To prove Theorem~\ref{thm:AF-ext}, we show next that if $E$ is row-finite and $I$ is a
gauge-invariant ideal of $C^*(E)$ such that $C^*(E)/I$ is AF, then $0 \to I \to C^*(E)
\to C^*(E)/I \to 0$ is quasidiagonal.

We thank James Gabe for providing us with the following result.

\begin{thm}[Gabe]\label{t:graphquasidiag-rf}
Let $E$ be a row-finite graph, $H$ a hereditary and saturated set such that the quotient
graph $E / H = ( E^{0} \setminus H , r^{-1} ( E^{0} \setminus H ) , r, s )$ has no cycles.
Then $0 \to I_{H}\to C^{*}(E) \to C^{*}(E)/I_{H} \to 0 $ is a quasidiagonal extension.
\end{thm}
\begin{proof}
Let $\{ p_{v} , s_{e} : v \in E^{0} , e \in E^{1} \}$ be the universal generating
Cuntz-Krieger $E$-family in $C^{*}(E)$.  Define
\begin{align*}
F_{1}(H) = H \cup \{\alpha \in E^{*} : \text{$\alpha = e_{1} \cdots e_{n}$
					 with $r( e_{n} ) \in H$, $s( e_{n} ) \notin H$}\}.
\end{align*}
Recall that if $F$ is a directed graph, then any increasing sequence of finite subsets
$V_n$ of $F^0$ such that $\bigcup V_n = F^0$ yields an approximate identity $\big(\sum_{v
\in V_n} p_v\big)_n$ of projections in $C^*(F)$.

We construct an approximate identity $\{e_X : X\text{ is a finite subset of }E^0\}$ for
$I_H$ that is quasicentral in $C^{*}(E)$ as follows.  By
\cite[Theorem~5.1]{RuizTomforde:xx12}, the elements $\{Q_\alpha : \alpha \in F_1(H)\}$
given by
\begin{align*}
Q_{\alpha} =
\begin{cases}
p_{v} &\text{if $v \in H$} \\
s_{\alpha} s_{\alpha}^{*} &\text{if $\alpha \in F_{1} (H)$}
\end{cases}
\end{align*}
are mutually orthogonal projections in $I_{H}$ and any sequence of increasing sums of these
which eventually exhausts all $\alpha \in F_{1} (H)$ is an approximate identity of projections
for $I_{H}$.

For a finite $X \subseteq E^0$, let
\begin{align*}
F_{1}(H)_X =
    (X \cap H) \cup \{\alpha \in F_{1}(H) : \alpha = e_{1} \cdots e_{k}
        \text{ and }s(e_1), \dots, s(e_k) \in X \setminus H\}.
\end{align*}
Since $E$ is row-finite and $E/H$ has no cycles, $F_1(H)_X$ is finite.  We have $\bigcup_X
F_{1}(H)_X = F_{1}(H)$. For $X \subseteq E^0$, set
\begin{align*}
e_X = \sum_{\alpha \in F_1(H)_X} Q_\alpha.
\end{align*}
Since $\bigcup_X F_1(H)_X = F_1(H)$, \cite[Theorem~5.1]{RuizTomforde:xx12} implies that
the $e_X$ (ordered by set-inclusion on finite subsets $X$ of $E^0$) constitute an
approximate identity for $I$ consisting of projections. We show that it is quasicentral.

Let $\beta, \gamma \in E^{*}$ with $r( \beta )  = r( \gamma ) = w$.  If $w \in H$, then
$s_{ \beta } s_{\gamma}^{*} \in I_{H}$, and hence
\begin{align*}
\lim_X \| e_X s_{ \beta } s_{\gamma}^{*} - s_{ \beta } s_{\gamma}^{*} e_X \| = 0.
\end{align*}
Suppose $w \notin H$.  Note that $p_{v} s_{\beta} s_{ \gamma}^{*} = s_{\beta}
s_{\gamma}^{*} p_{v} = 0$ for $v \in H$.  Let $\alpha \in F_{1} (H)$.  Since $r(\alpha)
\in H$ and $r( \beta ) = r( \gamma ) = w \notin H$, we have
\begin{align*}
s_{\alpha} s_{\alpha}^{*} s_{\beta} s_{\gamma}^{*} =
\begin{cases}
s_{\beta\alpha'} s_{\gamma\alpha'}^{*} &\text{if $\alpha = \beta \alpha'$} \\
0			&\text{otherwise}
\end{cases}
\quad \text{and} \quad
s_{\beta} s_{ \gamma }^{*} s_{\alpha} s_{\alpha}^{*} =
\begin{cases}
s_{\beta \alpha'} s_{\gamma \alpha'}^{*} &\text{if $\alpha = \gamma \alpha'$} \\
0 	&\text{otherwise}
\end{cases}
\end{align*}

Choose a finite $X_0 \subseteq E^0$ such that $r(\beta_i), s(\beta_i), r(\gamma_i)$ and
$s(\gamma_i)$ belong to $X$ for all $i$. For any $X$ containing $X_0$, we have $\beta
\alpha' , \gamma \alpha' \in F_{1}(H)_X$ for each $\alpha'\in F_{1}(H)_X$ with
$r(\alpha') = w$. Hence,
\begin{align*}
e_X s_{\beta} s_{ \gamma }^{*}
	= \sum_{ \alpha' \in F_{1} (H) , r( \alpha' ) = w }
		s_{\beta \alpha'} s_{ \gamma \alpha'}^{*} = s_{ \beta} s_{ \gamma }^{*} e_X.
\end{align*}
The claim now follows since $C^{*} (E) = \clsp\{ s_{\beta} s_{\gamma}^{*} : \beta, \gamma
\in E^{*} , r(\beta) = r(\gamma)\}$.
\end{proof}

\begin{cor}\label{cor:ndC*(E)=nd(I)}
Let $E$ be a graph and let $I$ be a gauge-invariant ideal of $C^{*}(E)$ such that
$C^{*}(E)/I$ is an AF-algebra.  Then $\nucdim(C^*(E)) = \nucdim(I)$.
\end{cor}
\begin{proof}
Let $F$ be a Drinen-Tomforde desingularization of $E$ \cite{DrinenTomforde:RMJM05}, so
that $C^*(F)$ is stably isomorphic to $C^*(E)$ and $F$ is row-finite with no sinks.
Rieffel induction over the Morita
equivalence coming from a Drinen-Tomforde desingularization carries gauge-invariant
ideals to gauge-invariant ideals \cite{DrinenTomforde:RMJM05}. So the ideal $J$ of
$C^*(F)$ corresponding to $I$ is gauge invariant and is stably isomorphic to $I$. The
quotient $C^*(F)/J$ is stably isomorphic to the AF algebra $C^*(E)/I$ and therefore
itself AF. Corollary~2.8(i) of \cite{WinterZacharias:AM10} implies that $\nucdim(I) =
\nucdim(I \otimes \Kk) = \nucdim(J \otimes \Kk) = \nucdim(J)$.
Theorem~\ref{t:graphquasidiag-rf} implies that $0 \to J \to C^*(F) \to C^*(F)/J \to 0$ is
quasidiagonal, and so Proposition~\ref{p:nucleardim-quasi} implies that $\nucdim(C^*(F))
= \max\{\nucdim(C^*(F)/J), \nucdim(J)\}$. Since $C^*(F)/J$ is AF, Remark~2.2 of
\cite{WinterZacharias:AM10} (see also Example~4.1 of \cite{KirchbergWinter}) implies that
$\nucdim(C^*(F)/J) = 0$, and we deduce that $\nucdim(C^*(F)) = \nucdim(J)$. To finish
off, we observe as above that since $C^*(F)$ and $C^*(E)$ are stably isomorphic, they
have the same nuclear dimension.
\end{proof}

\begin{proof}[Proof of Theorem~\ref{thm:AF-ext}]
The result follows directly from Corollary~\ref{cor:ndC*(E)=nd(I)} combined with
Corollary~\ref{cor:old pi nucdim} and Theorem~\ref{thm:pi graph alg}.
\end{proof}

\end{document}